\newcommand{\RR}{{R}}

\documentclass[makeidx, hyperindex, oneside, notitlepage,12pt]{article}

\usepackage[colorlinks=true,urlcolor=blue, citecolor=red, 
linkcolor=blue, bookmarks=true]{hyperref}

\usepackage{xcolor}

\usepackage{makeidx}

\makeindex

\newcommand{\dfn}[2]{{\it #1}{\index{#2}}} 

\usepackage{amssymb}
\usepackage[leqno]{amsmath}
\usepackage{amsfonts}
\usepackage{amsopn}
\usepackage{amstext}
\usepackage{amsthm}
\usepackage{amsthm}
\usepackage{bm}
\usepackage{mathrsfs}

\usepackage{tikz}
\usetikzlibrary{cd}

\usepackage{verbatim}

\usepackage{makeidx} 
\makeindex

\usepackage{color}

\usepackage{calrsfs}
\usepackage{fourier-orns}

\catcode `@ = 11
\renewcommand\section{\@startsection{section}{1}{\z@}%
                                   {-3.5ex \@plus -1ex \@minus -.2ex}%
                                   {2.3ex \@plus.2ex}%
                                   {\normalfont\large\bfseries}}
\renewcommand\subsection{\@startsection{subsection}{2}{\z@}%
                                     {-3.25ex\@plus -1ex \@minus -.2ex}%
                                     {1.5ex \@plus .2ex}%
                                    {\normalfont\bfseries}} 
\renewcommand\subsubsection{\@startsection{subsubsection}{2}{\z@}%
                                     {-3.25ex\@plus -1ex \@minus -.2ex}%
                                     {1.5ex \@plus .2ex}%
                                    {\normalfont\itshape}} 
\catcode `@ = 12

\providecommand{\mcal}{\mathcal}
\renewcommand{\Bbb}{\mathbb}

\newenvironment{pf}{\begin{proof}}{\end{proof}}

\newcommand{\Pee}{{\mcal{P}}}

\newcommand{\Yu}{{\mcal{U}}}
\newcommand{\Vee}{{\mcal{V}}}

\newcommand{\calY}{{\mcal{Y}}}

\newcommand{\Be}{{\Bbb{B}}}

\newcommand{\Qyu}{{\Bbb{Q}}}

\newcommand{\natA}{{\mathbb A}}
\newcommand{\natB}{{\mathbb B}}

\newcommand{\natM}{{\mathbb M}}
\newcommand{\natN}{{\mathbb N}}

\newcommand{\natQ}{{\mathbb Q}}
\newcommand{\natR}{{\mathbb R}}

\def\calU{\mathcal{U}}
\def\calV{\mathcal{V}}
\def\calW{\mathcal{W}}

\newcommand{\lam}{{\lambda}}
\newcommand{\al}{\alpha}

\newcommand{\sig}{\sigma}
\newcommand{\eps}{\varepsilon}
\renewcommand{\phi}{\varphi}
\renewcommand{\rho}{\varrho}

\newcommand{\rest}{\restriction}

\newcommand{\ntr}{{n\in\omega}}

\newcommand{\loe}{\leqslant}
\newcommand{\goe}{\geqslant}

\newcommand{\subs}{\subseteq}
\newcommand{\sups}{\supseteq}
\newcommand{\nnempty}{\ne\emptyset}

\newcommand{\cl}{\operatorname{cl}}

\newcommand{\w}{\operatorname{\mathbf{w}}}

\newcommand{\cf}{\operatorname{cf}}

\newcommand{\supp}{\operatorname{supp}}

\newcommand{\oraz}{\qquad\text{and}\qquad}

\newcommand{\meet}{\wedge}

\newcommand{\join}{\vee}

\newcommand{\by}{/}

\newtheorem{tw}{Theorem}[section]
\newtheorem{wn}[tw]{Corollary}
\newtheorem*{corollary*}{Corollary}
\newtheorem{lm}[tw]{Lemma}
\newtheorem{prop}[tw]{Proposition}
\newtheorem*{proposition*}{Proposition}
\newtheorem*{theorem*}{Theorem}
\newtheorem{claim}[tw]{Claim} 

\newtheorem{corollary}[tw]{\bf Corollary}
\newtheorem{proposition}[tw]{\bf Proposition}
\newtheorem{lemma}[tw]{\noindent {\bf Lemma}}

\theoremstyle{definition}
\newtheorem{df}[tw]{Definition}
\newtheorem*{comment*}{Comment}
\newtheorem*{remark*}{Remark}
\newtheorem*{Observation*}{\bf Observation}

\newtheorem{pyt}[tw]{Question}

\newtheorem{uwgi}[tw]{Remark}

\newcommand{\setof}[2]{\{#1\colon #2\}}

\newcommand{\sett}[2]{\{#1\}_{#2}}
\newcommand{\sn}[1]{\{#1\}} 
\newcommand{\dn}[2]{\{#1,#2\}} 
\newcommand{\pair}[2]{\langle #1, #2 \rangle} 
\newcommand{\triple}[3]{\langle #1, #2, #3 \rangle} 
\newcommand{\fourthple}[4]{\langle #1, #2, #3, #4 \rangle} 
\newcommand{\map}[3]{#1\colon #2 \to #3} 
\newcommand{\img}[2]{#1[#2]} 
\newcommand{\inv}[2]{{#1}^{-1}[#2]} 

\newcommand{\fin}[1]{[#1]^{<\omega}}

\newcommand{\norm}[1]{\|#1\|}

\newcommand{\abs}[1]{|#1|}

\newcommand{\cmp}{\circ} 

\newcommand{\cont}{\ensuremath{\mathfrak c}}

\newcommand{\separator}{\begin{center} \leafright \leafright \leafright \decotwo \decotwo \decotwo \leafleft \leafleft\leafleft
\end{center}}

\def\powerset{{\mathscr{P}}}  

\newcommand{\rfs}[1]{{\rm #1\kern 0.7pt}}

\newcommand{\trpl}[3]{\langle #1, #2, #3 \rangle}

\newcommand{\fnn}[3]{#1:#2 \rightarrow #3}
\newcommand{\cat}{\kern 1.8pt{\mathaccent 94 {\,\ }}\kern 0.7pt}
\newcommand{\onetpl}[1]{\langle \kern1pt #1 \kern1pt \rangle}

\newcommand{\seqnn}[2]{\langle#1\rangle_{#2} }

\newcommand{\fourtpl}[4]{\langle\kern1pt#1, #2, #3, #4\kern1pt\rangle}
\newcommand{\qdrpl}[4]{\langle\kern1pt#1, #2, #3, #4\kern1pt\rangle}
\newcommand{\fivtpl}[5]{\langle\kern1pt#1, #2, #3, #4, #5\kern1pt\rangle}
\newcommand{\sixpl}[6]{\langle\kern1pt#1, #2, #3, #4, #5,
#6\kern1pt\rangle}

\def\join{\vee}

\def\meet{\wedge}

\newcommand{\symdiff}{\mathbin{\mbox{\scriptsize$\, \triangle \, $}}}

\newcommand{\upclcone}[1]{[\,{#1},\rightarrow)}

\def\Clop{{\rm Clop}}
\def\clop{{\rm Clop}}

\def\Ult{{\rm Ult}}
\def\ult{{\rm Ult}}

\newcommand{\Fr}{\operatorname{Fr}}

\def\eqdef{\mbox{\bf\ :=\ }}  

\newcommand{\BEGINPROOF}[1]%
{\noindent{\it Proof of #1. }}

\newcommand{\band}[1]{\ensuremath{\mathfrak B_{#1}}}

\newcommand{\ord}{\operatorname{ord}}

\usepackage{makeidx}

\newcommand{\rloe}{\preceq}

\newcommand{\alex}{\operatorname{D}}
\newcommand{\tight}{\operatorname{\mathfrak{t}}}

\newcommand{\koment}[1]{}

\title{\bf Ultrafilter selection and Corson compacta}

\author{Robert Bonnet%
\footnote{
Laboratoire de Math\'ematiques,
Universit\'e de Savoie et Mont Blanc, Le Bourget-du-Lac, France.
The first author was supported by the Institute of Mathematics of the Czech Academy of Sciences, Prague.}
, Wies{\l}aw Kubi\'s%
\footnote{
Institute of Mathematics,
Czech Academy of Sciences,
Praha, Czech Republic. The second author was supported by the GA \v CR grant 20-22230L (Czech Science Foundation).}
, and Stevo Todor\v cevi\'c%
\footnote{
Institut de Math\'ematiques de Jussieu, Paris, France, and Department of Mathematics, University of Toronto, Canada.
}
}
{\date{ 
\today}}

\makeindex

\begin{document}

\thispagestyle{empty}

\maketitle

\begin{abstract}
We study the question which Boolean algebras have the property that for every generating set there is an ultrafilter selecting maximal number of its elements.
We call it the \emph{ultrafilter selection property}.
For cardinality $\aleph_1$ the property is equivalent to the fact that the space of ultrafilters is not Corson compact. 
\koment{We generalize Corson compact spaces for uncountable support and we show that almost properties of this class are preserved.}
We also consider the pointwise topology on a Boolean algebra, proving a result on the Lindel\"of number in the context of the ultrafilter selection property.
Finally, we discuss poset Boolean algebras, interval algebras, and semilattices in the context of ultrafilter selection properties.
\end{abstract}

\koment{\vfill
	\noindent
	\begin{footnotesize}%
		Note that results \ref{fact-uiuio}, \ref{cor-ieyzuirzeuirez} and \ref{cor-cor-456456456} are news and almost obvious. 
		See also the questions after \ref{prop-sigma-2}.
		\\ 
		The corrections are only suggestions for \bk{comments} or \BK{suggestion}: 
	\end{footnotesize}
}

\vfill

\noindent%
\begin{footnotesize}%
{\bf Mathematics Subject  Classification (2010)}
{\bf Primary:}
06E15, 
54D30. 
{\bf Secondary:}
03G10, 
54A25,  
\newline
\end{footnotesize}

\newpage


\phantom{bk}

\vspace{-5.5em}
\tableofcontents 
\bigskip 

{\bf References (and Index) \hfill \pageref{references}}

\newpage

\section{Introduction}
\label{introduction}

Given an infinite Boolean algebra $\Be$, we may ask whether for every generating set $G$ there is an ultrafilter $p$ satisfying $|p \cap G| = |\Be|$.
In this case we call it the \dfn{ultrafilter selection property}{ultrafilter selection property}. 
There are many Boolean algebras which drastically fail this property: For example, let $\Be$ be the subalgebra of $\Pee(\kappa)$ generated by finite sets (the finite-cofinite algebra).
Taking $G$ consisting of all singletons, we see that $|p \cap G| \loe 1$ for every ultrafilter $p$ on $\Be$.

In Theorem~\ref{ThmTrojkaa} below we prove that every poset algebra (in particular, every interval algebra) of uncountable regular cardinality has the ultrafilter selection property.
Recall that, given a chain (i.e., a linearly ordered set) $C$, the \dfn{interval algebra}{interval algebra} $\Be(C)$ over $C$ is the subalgebra of $\powerset(C)$ generated by all intervals of the form 
$$[a,\rightarrow) \eqdef \setof{ x \in C}{ x \geq a},$$ where $a \in C$.
For the definition of a poset Boolean algebra, see Section~\ref{SectModestlats}.

In this note we are rather interested in the negation of the ultrafilter selection property. Namely, a Boolean algebra $\Be$ is called \dfn{$\kappa$-Corson}{$\kappa$-Corson} if there exists a generating set $G \subs \Be$ such that $|G \cap p| < \kappa$ for every $p \in \ult (\Be)$.
We shall see later that $\Be$ is $\kappa$-Corson if and only if its space of ultrafilters is $\kappa$-Corson (see Definition~\ref{DFkapaCorsn} below).
In fact, $\Be$ is $\aleph_1$-Corson iff the compact space $\ult (\Be)$ is a Corson compact.
The notion of a $\kappa$-Corson compact was introduced by Kalenda~\cite{Ka1} and later studied by Bell and Marciszewski~\cite{BeMa} for infinite successor cardinals. 
The definitions above make sense also for $\kappa = \aleph_0$, however in that case the situation is somewhat clear:

\begin{prop}[\cite{Alster}]
\label{prop-1212}
An Boolean algebra is $\aleph_0$-Corson if and only if its Stone space is a strong Eberlein compact space.
\end{prop}

In particular, for a countable Boolean algebra it turns out that being $\aleph_0$-Corson is equivalent to being superatomic.
A space $K$ is a \dfn{strong Eberlein compact}{strong Eberlein compact} if there is a $T_0$-separating point finite family consisting of open $F_\sig$ sets (see Proposition~\ref{Propeholnjewr} below). Strong Eberlein compacta were introduced and studied by Alster~\cite{Alster}.

A natural (and still useful) generalization of $\kappa$-Corson algebras is the following: We say that a Boolean algebra $\Be$ is \dfn{$\kappa$-Valdivia}{$\kappa$-Valdivia} if there exist a generating set $G \subs \Be$ and a topologically dense set $D \subs \ult(\Be)$ such that $|G \cap p| < \kappa$ for every $p \in D$.
This notion was introduced by Kalenda~\cite{Ka1} and, for infinite successor cardinals, was also studied by Bell and Marciszewski~\cite{BeMa}.
Note that every free Boolean algebra is $\aleph_0$-Valdivia.
In fact, taking the free algebra $\Fr(S)$ generated by $S$, 
the set $\setof{p \in \ult{(\Fr(S)})}{|p \cap S| < \aleph_0}$ is topologically dense. 
Indeed, let $W$ be a basic nonempty clopen set of $2^S$. 
So there are finite disjoint sets $\sigma, \tau \subseteq S$ such that 
$W = \prod \sigma \cdot \prod_{t \in \tau} -t$.
Then $\sigma \cup \setof{-r}{r \in S \setminus \sigma}$ extends to a unique ultrafilter $q$, and $| q \cap S | = |\sigma|<\aleph_0$.

In order to illustrate the rather delicate nature of the ultrafilter selection properties we present the following simple yet somewhat nontrivial statement.

\begin{prop}
\label{lemma-jkljlkjlk}
For every generating set $G$ of the interval algebra $\Be(\natR)$,
the set $\setof{ p \in \Ult(\Be(\natR)) }{ | G \cap p | < 2^{\aleph_0} }$
is nowhere dense in $\Ult(\Be(\natR))$.
\end{prop}

\begin{pf}
	Fix a generating set $G$ in $\Be \eqdef \Be(\natR)$.
	Fix a countable set $D \subs \Be$ consisting of positive elements and dense in the algebraic sense, i.e., for every $a \nnempty$ there is $d \in D$ with $d \subs a$. For example, $D$ could be the set of all nonempty intervals $[s,t)$ with rational end-points.
	
Fix $b \in \Be$, $b \nnempty$. 
	Note that the set $G_b \eqdef \setof{g \cap b}{g\in G}$ generates the relative algebra $\Be \rest b$.
	In particular, $|G_b| = 2^{\aleph_0}$.
	For each $g \in G$ such that $g \cap b \nnempty$ choose $d(g) \in D$ with $d(g) \subs g \cap b$.
	There are $H \subs G$ and $e \in D$ such that $|H| = 2^{\aleph_0}$ and $d(h) = e$ for every $h \in H$.
	Finally $e \subs b$, and $|G \cap p| = 2^{\aleph_0}$ whenever $p \in \ult(\Be)$ is such that $e \in p$. 
\end{pf}

This shows that $\Be(\natR)$ is far from being $\cont$-Valdivia.
It turns out to be quite non-trivial to prove that a subalgebra of a $\kappa$-Corson algebra is $\kappa$-Corson
(Corollary~\ref{Thmjgwrgo}).
In fact, we are able to show it only in the case where $\kappa$ is a regular cardinal.

The names ``$\kappa$-Corson" and ``$\kappa$-Valdivia" are inspired by existing classes of Corson and Valdivia compact spaces, where $\kappa = \aleph_1$.
Usually, these classes are defined topologically. 
We explain it in the next section.

\separator

The present note is devoted to the study of $\kappa$-Corson and $\kappa$-Valdivia Boolean algebras, where $\kappa$ is an arbitrary uncountable regular cardinal. Such a study has already been made by Bell and Marciszewski~\cite{BeMa} for $\kappa$ being a successor cardinal. We extensively use the method of stable elementary submodels, characterizing $\kappa$-Corson algebras and their Stone spaces, extending Bandlow's characterization~\cite{Bandlow2} of Corson compacta. As an application, we show that $\kappa$-Corson Boolean algebras are closed under subalgebras, equivalently, $\kappa$-Corson compacta are closed under continuous images.
We also explore concrete examples, like free Boolean algebras over posets and, in particular, interval algebras. Finally, we study the natural pointwise convergence topology on the Boolean algebra, induced by the Stone duality.

\section{Corson-like compact spaces}
\label{section-corson-like}

In this section we sketch some results and arguments in general case (not necessarilly non-zero-dimensional setting).
Namely, Corson-like properties of general compact spaces, where Boolean algebras are not applicable.
Actually, the starting point is Proposition~\ref{Propeholnjewr}, where the notion of being $\kappa$-Corson compact originates.
We would like to characterize this property in terms of elementary submodels.
Actually, the case $\kappa=\aleph_1$ was treated by Bandlow~\cite{Bandlow1, Bandlow2, Bandlow3}.
Below we consider spaces of the form $C(K)$,
consisting of all real-valued continuous functions on a compact space $K$, always endowed with the pointwise convergence topology $\tau_p$. If $K$ is zero-dimensional, we shall also consider $C(K,2)$, consisting of all characteristic functions of clopen sets, which is clearly a closed subspace of $C(K)$. Note that $C(K)$ is a Banach space (endowed with the maximum norm), however we will be interested in the pointwise topology $\tau_p$ only.
We shall use the well known fact that given a continuous surjection $\map f K L$ between compact spaces, the dual operator $\map{f^*}{C(L)}{C(K)}$ is a topological embedding with respect to pointwise convergence topologies. It is also an isometric linear embedding, when $C(K)$ and $C(L)$ are treated as Banach spaces.

Given a set $S$, define
$$\Sigma_\kappa(S) = \setof{x \in [0,1]^S}{|\supp(x)| < \kappa},$$
where $\supp(x) \eqdef \setof{ \al \in S }{ x(\al) \ne 0 }$ is the \dfn{support}{support} of $x$. We consider $\Sigma_\kappa(S)$ endowed with the pointwise topology.

Given a family of sets $\Yu$ we define
$\ord(x,\Yu) = |\setof{U \in \Yu}{x \in U}|$.
This is usually called the \dfn{order}{order of a point} of $x$ with respect to $\Yu$.

\begin{df}\label{DFkapaCorsn}
	We say that a compact space $K$ is \dfn{$\kappa$-Corson}{$\kappa$-Corson} whenever $K$ is topologically embeddable in $\Sigma_{\kappa}(S)$ for some set $S$. 
	So Corson compact spaces correspond to $\aleph_1$-Corson spaces.
\end{df}

\begin{prop}
\label{Propeholnjewr}
Let $\kappa > \aleph_0$ be a cardinal.
For a compact space $K$, the following conditions are equivalent:
\begin{enumerate}
	\item[\rm(i)] The space $K$ is $\kappa$-Corson.
	\item[\rm(ii)] There exists a family $\Yu$ consisting of open $F_\sig$ subsets of $K$ which is $T_0$ separating, i.e., for every $x,y \in K$ there is $U \in \Yu$ satisfying $|U \cap \dn x y| = 1$, and $\ord(x,\Yu) < \kappa$ for each $x \in K$.
\end{enumerate}
Furthermore, if $K$ is $0$-dimensional, the family $\Yu$ from condition {\rm(ii)} may be assumed to consist of clopen sets.

Therefore any Boolean algebra $\natB$ is 
$\kappa$-Corson if and only if its space 
$\Ult(\natB)$ of ultrafilters is $\kappa$-Corson.
\end{prop}

\begin{pf}
(i)$\implies$(ii) 
Assume $K \subs \Sigma_\kappa(S)$ and
let $$\Yu = \setof{\pi_\al^{-1}(r,1]}{\al \in S, \; r \in [0,1] \cap \Qyu},$$ 
where $\map {\pi_\al} K {[0,1]}$ is the projection of $K$ onto the $\al$-th coordinate.
It is clear that $\Yu$ satisfies (b).

(ii)$\implies$(i) For each $U \in \Yu$ choose a continuous function $\map {f_U}K{[0,1]}$ such that $f_U^{-1}(0,1] = U$.
Since $\Yu$ is $T_0$ separating, the diagonal map $\map f K {[0,1]^\Yu}$, defined by $f(x)(U) = f_U(x)$, is a continuous  embedding.
Clearly, $\img f K \subs \Sigma_\kappa(\Yu)$.

Finally, if $K$ is $0$-dimensional and $\Yu$ satisfies condition (ii), then we can replace each $U \in \Yu$ by a countable collection of clopen sets $\Vee_U$ such that $\bigcup \Vee_U = U$. 
Obviously $\Vee = \bigcup_{U \in \Yu} \Vee_U$ satisfies condition (ii).
\end{pf}

We continue with a rather trivial observation for which the proof is left to the reader.

\begin{prop}
\label{prop-yfsdio}
Let $\kappa$ be an infinite cardinal.
\begin{enumerate}

\item[\rm{(1)}] 
The product of\/ ${<}\cf(\kappa)$ many \,$\kappa$-Corson spaces is $\kappa$-Corson. 

\item[\rm{(2)}] 
The one-point compactification of a direct sum of  $\kappa$-Corson spaces is $\kappa$-Corson.

\item[\rm{(3)}] 
Every compact subspace of a $\kappa$-Corson space is $\kappa$-Corson. 
\qed
\end{enumerate}
\end{prop}

A topological space is called \dfn{$\kappa$-Valdivia compact}{Valdivia compact}
if it is homeomorphic to a closed subspace $K$ of a cube $[0,1]^S$ such that
$D \eqdef K \cap \Sigma_\kappa(S)$ is dense in $K$.  So $D=K$ if and only if $K$ is $\kappa$-Corson.

A similar characterization to Proposition~\ref{Propeholnjewr}, involving a dense set, is true for $\kappa$-Valdivia compact spaces.

\subsection{Bandlow characterization of Corson-like spaces}
\label{Corson-like}

For set-theoretic notions, we follow Kunen~\cite{Ku}. 
Concerning the method
of elementary submodels we refer to \cite{Dow}, \cite[Ch. 24]{JW} and \cite[Ch. 17.1]{KKLP}.
Concerning
Corson and Valdivia compact spaces we refer to~\cite[Ch.~19]{KKLP}, Kalenda~\cite[Ch. 1]{Ka1} and \cite{Ka2}. 
In particular, $\aleph_1$-Corson and $\aleph_1$-Valdivia compact spaces are called 
Corson and Valdivia compact spaces.

Recall that for a space $K$ and an elementary submodel $M$ of some $H(\theta)$,  $K \in M$ means that both the set $K$ and {a base of} its topology $\tau$ are elements of $M$. 
Furthermore, if $K \in M$
then, by elementarity, $C(K)$ with its pointwise topology is also in $M$, because it is uniquely defined from parameters in $M$, namely by open intervals with rational endpoints. 

Following Bandlow~\cite{Bandlow2}, given an elementary submodel $M$ of some $H(\theta)$, given a compact 
space $K \in M$, we define the following equivalence relation on $K$ 
as follows: 
$$x \sim_M y \text{ \ whenever  for every  \ } f \in C(K) \cap M \text{  \ it holds that \ } f(x) = f(y).$$ 
Denote by $K \by M$ the quotient space with respect to this relation. 

Let us prove the following useful property of $\sim_M$.

\begin{lm}
\label{Lmeruvggh}
Let $K$ be a compact space and let $M$  be an  elementary submodel of some $H(\theta)$ satisfying $K \in M$.
Assume $A,B \subs K$ are closed sets such that $a \not \sim_M b$ whenever $a \in A$, $b \in B$.
Then there exists $h \in C(K) \cap M$ such that $h(a) = 0$ for every $a \in A$ and $h(b) = 1$ for every $b \in B$. 

Furthermore, if $K$ is 0-dimensional then we may assume that $h$ takes values in the set $\{0,1\}$. In other words, in this case there exists a clopen set $U \in M$ with $A \cap U = \emptyset$ and $B \subs U$.
\end{lm}

\begin{pf}
Fix $a \in A$. For each $y \in B$ choose $f_y \in C(K) \cap M$ such that $f_y(a) = 0$ and $B \subs f_y^{-1}(1)$.
The family $\setof{\inv {f_y} {(1/2,1]}}{y \in B}$
 is an open cover of $B$, therefore it has a finite subcover, leading to a function $g \in C(K) \cap M$ such that $g(a) = 0$ and $g(y) > 1/2$ for every $y \in B$.
Let us denote the function $g$ as $g_a$, since it depends on the choice of $a \in A$.
Using the same argument, 
with $(1/2, 1]$ replaced by $[0,1/3)$,
 for the family $\sett{g_a}{a\in A}$, we get a function $h \in C(K) \cap M$ such that $h(a) < 1/3$ for every $a\in A$ and $h(b) > 1/2$ for every $b \in B$.
 Finally, replace $h$ by the composition $k \cmp h$, where $k \in M$ maps the interval $(1/2,1]$ onto $\sn1$ and maps the interval $[0,1/3)$ onto $\sn0$.
 
 If $K$ is 0-dimensional, we can repeat the same arguments replacing $[0,1]$ by $\{0,1\}$.
\end{pf}

If $M$ is an elementary submodel of some $\pair{H(\theta)}{\in}$ and $\kappa$ is a cardinal, then $\kappa \cap M$ is a set of ordinals but $\kappa \cap M$ is not in general an ordinal. 
So we introduce the following notion.

\begin{df}
	{An elementary submodel $M$ of some $\pair{H(\theta)}{\in}$ such that $\kappa \in M$ and $\kappa \cap M$ is an initial segment of $\kappa$ will be called \dfn{$\kappa$-stable}{stable submodel}.}
\end{df}

Here is perhaps the most important consequence of being $\kappa$-stable.

\begin{prop}\label{PROPsidvod}
	Assume $M \rloe H(\theta)$ is $\kappa$-stable, where $\kappa$ is a regular uncountable cardinal. Then
	 for every $A \in M$ such that $|A| < \kappa$ it holds that $A \subs M$.
\end{prop}

\begin{pf}
Let $\lam := |A| < \kappa$. By elementarity, $\lam \in M$ and thus also $\lam+1 \in M$. 
Therefore $\lam \subs M$ because $M$ is $\kappa$-stable.
	Again by elementarity, there exists a bijection $f \in M$ from $\lam$ onto $A$. Finally, each $a \in A$ is of the form $f(\al)$ for some $\al \in \lam$. 
\end{pf}

The following rather trivial consequence of the above result will be used several times without explicit reference.

\begin{corollary}\label{CORsidvod}
 Let $K \subseteq \Sigma_\kappa(S)$ be a $\kappa$-Corson compact space and $M$ be a $\kappa$-stable submodel of a big enough  $H(\theta)$ such that $K, S \in M$ and $|M|<\kappa$.
	If $x \in K \cap M$ then $\supp(x) \subseteq S \cap M$. 
	\qed
\end{corollary}

The existence of $\kappa$-stable elementary substructures of cardinality $<\kappa$ is almost trivial in case $\kappa \eqdef \lambda^+$ is a successor cardinal. Namely, if $\kappa = \lam^+$ then every $M \rloe H(\theta)$ such that $\lam+1 \subs M$ is $\kappa$-stable, by the proof of Proposition~\ref{PROPsidvod}.
The general case follows from the following:

\begin{lemma}
\label{LmStlrbggnkn}
Let $\kappa < \theta$ be regular cardinals
and 
let $A \in H(\theta)$ be a set of cardinality $< \kappa$.
Then there exists a $\kappa$-stable elementary submodel $M$ of $H(\theta)$ such that $|M| < \kappa$ and $A \in M$.
\end{lemma}

Note that, by Proposition~\ref{PROPsidvod}, we also have $A \subs M$.
\begin{pf}
	Using the downward L\"owenheim-Skolem Theorem, we can build a chain $\sett{M_n}{\ntr}$ of elementary submodels of $H(\theta)$, each of cardinality $<\kappa$, such that $M_0$ is countable, $A \in M_0$, $M_{n+1}$ contains $M_n \subs \delta_n$, where $\delta_n = \sup (M_n \cap \kappa)$. By the regularity of $\kappa$, $\delta_n < \kappa$
as long as $|M_n| < \kappa$. Finally, $M := \bigcup_{\ntr}M_n$ is an elementary submodel of $H(\theta)$, its cardinality is $<\kappa$ and $M \cap \kappa = \bigcup_{\ntr} \delta_n$.
\end{pf}

We are now ready to introduce Property $\band \kappa$.

\begin{df}
\label{dfn-bk}
{Let $\kappa$ be an uncountable regular cardinal.
We say that a compact space $K$ \dfn{has Property $\band \kappa$}{Property $\band \kappa$ (for compact spaces)} if for every big enough $\theta$, for every $\kappa$-stable elementary submodel $M$ of $\pair {H(\theta)}{\in}$ such that $K \in M$ and $|M| < \kappa$, the equivalence relation $\sim_M$ is one-to-one on the set $\cl(K \cap M)$.}
\end{df}

The main result of this section is to prove the following theorem.

\begin{tw}\label{thm-753}
Let $\kappa$ be an uncountable regular cardinal.
A compact space is $\kappa$-Corson if and only if it has Property $\band \kappa$.
\end{tw} 

To prove this result we need some preliminary facts.
Given $p \in \cl(A)$
in a topological space, we define
$$\tight(p, A) = \min \setof{|B|}{p \in \cl(B) \text{ and }B \subs A}.$$
We call it the \dfn{local tightness of $p$ with respect to $A$}{$\tight(p, A)$: local tightness}.
We shall say that a topological space $K$ has \dfn{local tightness}{local tightness $<\kappa$} $<\kappa$ if $\tight(p,A) < \kappa$ for any $p \in \cl(A) \subs K$.
It may happen, in case $\kappa$ is a limit cardinal, that a space with local tightness $< \kappa$ has tightness $\kappa$ (the supremum of all cardinals of the form $\tight(p,A)$ with $p \in \cl (A) \subs K$).

\begin{prop}
\label{prop-gsdfjgjhsdf}
	Let $K$ be a compact space with Property $\band{\kappa}$. 
	Then $K$ has local tightness $<\kappa$.
\end{prop}

\begin{pf} 
Fix $x \in \cl A \subs K$ and choose a $\kappa$-stable  elementary submodel $M$ of a big enough $H(\theta)$ such that $x, A, K \in M$ and $|M| < \kappa$.
	We claim that $x \in \cl(A \cap M)$.
	Suppose otherwise. 
	Since $K$ has Property $\band \kappa$ and $x \in M$, we can use Lemma~\ref{Lmeruvggh} in order to obtain $f \in C(K) \cap M$ such that $f(x) = 0$ and $f(a) = 1$ for every $a \in A \cap M$.
	But now $M$ thinks that $f$ separates $x$ from $A$, which means that $x \notin \cl A$, a contradiction.
\end{pf}

We show first the ``only if'' part of Theorem~\ref{thm-753}.

\begin{lm}\label{LMerighuerghvo}
Assume $\kappa$ is an uncountable regular cardinal and $K$ is a $\kappa$-Corson compact space.
Then $K$ has Property $\band \kappa$.
\end{lm}

\begin{pf}
	Assume $K \subs \Sigma_\kappa(S)$ and take $M \rloe H(\theta)$ such that $K,S \in M$.
	Then $\cl(K \cap M)$ is precisely the set of all $x \in K$ such that $\supp(x) \subs S\cap M$ (and so Property $\band \kappa$ holds for $K$).
	Indeed, if $x \in K \cap M$ then $\supp(x) \subs M$ because $M$ is $\kappa$-stable. 
	On the other hand, if $\supp(x) \subs S \cap M$ then, using elementarity, it is easy to see that every basic neighborhood of $x$ contains some element of $K\cap M$. 
	Indeed, a basic open set is determined by finitely many coordinates $\alpha_i \in S \cap M$ and open intervals with rational end-points on each coordinate $\al_i$.
\end{pf}

The next result is the key to prove the ``if'' part of Theorem~\ref{thm-753}. 

\begin{lm}
\label{lemma-uiouoi}
Let $\kappa$ be an uncountable regular cardinal, $K$ a compact space with Property $\band{\kappa}$, $M$ be a $\kappa$-stable elementary submodel of a big enough $H(\theta)$ with $K \in M$.
Then there is a retraction $\map {{\RR}_M} K K$ such that
\begin{enumerate}
	\item[$(1)$] $\img {{\RR}_M} K = \cl( K \cap M)$,
	\item[$(2)$] ${\RR}_M(x) = {\RR}_M(y)$ if and only if $x \sim_M y$, for every $x, y \in K$.
\end{enumerate}
\end{lm} 

\begin{pf}
	Fix $x \in K$.
	We claim that there is a unique $y \in \cl(K\cap M)$ such that $x \sim_M y$.
	Indeed, supposing that $x \not\sim_M y$ for every $y \in \cl(K \cap M)$, by Lemma~\ref{Lmeruvggh} we would have a function $f \in C(K) \cap M$ such that $f(x) = 0$ and $f(y) = 1$ for every $y \in \cl(K \cap M)$. Then $M \models$ ``$f = 1$ constantly'', which is a contradiction.
	This shows the existence of $y \in \cl(K \cap M)$ satisfying $x \sim_M y$. Uniqueness follows directly from the definition of $\sim_M$. 
Denote this unique $y$ by ${\RR}_M(x)$.
	Conditions (1) and (2) are clearly satisfied.
	It is also clear that ${\RR}_M(x) = x$ for every $x \in \cl(K \cap M)$.
	It remains to show that ${\RR}_M$ is continuous.
	
	Fix $x \in K$ and fix a neighborhood $V$ of ${\RR}_M(x)$ in $\cl(K \cap M)$.
	By Lemma~\ref{Lmeruvggh}, there is $f \in C(K) \cap M$ such that $f(x) = 0$ and $f(y) = 1$ for every $y \in \cl(K \cap M) \setminus V$.
	Let $U = f^{-1} [0,1/2)$. Then $U$ is a neighborhood of $x$ and since $U \in M$, we deduce that ${\RR}_M(z) \in U$ whenever $z \in U$, because ${\RR}_M(z) \sim_M z$ for every $z \in K$. 
	Thus ${\RR}_M[U] \subs U \cap \cl(K \cap M) \subs V$. This shows the continuity of ${\RR}_M$ and finishes the proof.
\end{pf}

\begin{comment*}
\label{comment-000}
\begin{rm}
{\em In what follows, we identify $K/M$ and $\cl(K \cap M)$} whenever $K$ is $\kappa$-Corson compact and $M$ is a suitable elementary submodel of a big enough $H(\theta)$.
More precisely by Lemma~\ref{lemma-uiouoi}, the map $\pi_M$ from $\cl(K \cap M)$ into $K/M$ defined by $\pi_M(x) = x/ {\sim_M}$ is a bijection  and a  homeomorphism onto whenever we endow $K/M$ with the quotient topology. 
\hfill$\blacksquare$
\end{rm}
\end{comment*}

The proof of the next lemma is an adaptation of the proof of Lemma 13 in Bandlow~\cite{Bandlow2}.

\begin{lm}
\label{LMnioheboier}
	Property $\band{\kappa}$ is closed under continuous images.
\end{lm}

\begin{pf}
Fix a continuous surjection $\map \phi K L$ and assume $K$ has Property $\band{\kappa}$.

Fix a big enough elementary $\kappa$-stable submodel $M  \rloe H(\theta)$ such that $\phi \in M$, and thus $K, L \in M$.
Fix $x \ne y$ in $\cl(L \cap M)$ and denote $K_0 = \cl(K \cap M)$.

Let $A = \phi^{-1}(x) \cap K_0$ and $B = \phi^{-1}(y) \cap K_0$.
Choose disjoint open sets $U$, $V$ in $L$ such that $x \in U$, $y \in V$.
Let $A' \eqdef K_0 \setminus \inv \phi U$.
Using Property $\band{\kappa}$ of $K$, note that $a \not \sim_M a'$ whenever $a \in A$, $a' \in A'$.
By Lemma~\ref{Lmeruvggh}, there is a continuous function $h \in M$ such that $\img h A = \sn0$ and $\img h{A'} = \sn1$.
Then the set $F \eqdef h^{-1}(0) \cap A$
 is nonempty, closed, belongs to $M$ and satisfies
$$\phi^{-1}(x) \cap K_0 = A \subs F \oraz F \cap K_0 \subs \inv \phi U.$$
Similarly, the set $H \eqdef h^{-1}(1) \cap  B$ is nonempty, closed, belongs to $M$ and satisfies
$$\phi^{-1}(y) \cap K_0 = B \subs H \oraz H \cap K_0 \subs \inv \phi V.$$
We claim that $\img \phi F \cap \img \phi H = \emptyset$.
Indeed, suppose $\img \phi F \cap \img \phi H \nnempty$. 
Using elementarity we find $p \in F \cap M$, $q \in H \cap M$ such that $\phi(p) = \phi(q)$.
Since $p,q \in K_0$, necessarily $\phi(p) \in U$ and $\phi(q) \in V$, which contradicts the fact that $U \cap V = \emptyset$.
So $\img \phi F$, $\img \phi H$ are disjoint closed subsets of $L$ that belong to $M$. 

Elementarity (together with Urysohn's Lemma) provides a continuous function $g \in M$ separating $\img \phi F$ from $\img \phi H$.
Finally, note that $x \in \img \phi F$, $y \in \img \phi H$ and hence $x \not\sim_M y$.
\end{pf}

\begin{pf}[Proof of Theorem {\rm{\ref{thm-753}}}]
We use induction on the weight $\w(K)$ of the space $K$. So, assume $K$ has Property $\band{\kappa}$ and the theorem is valid for compact spaces of weight $< \mu$. 

Since $K$ is embeddable in $[0,1]^{\w(K)}$, if $\w(K) < \kappa$, then $K$, as product of $<\kappa$ \ $\kappa$-Corson spaces,  is $\kappa$-Corson. 
So we may assume that $\w(K) = \mu \goe \kappa$. 

Let $\lam = \cf \mu$.
We fix a continuous chain $\sett{M_\al}{1 \loe \al \loe \lam}$ of $\kappa$-stable elementary submodels of a big enough $H(\theta)$ such that $K \in M_1$, $|M_\al| < \mu$ for $\al < \lam$ and $M_\lam$ contains a fixed base for the topology of $K$ (in particular, $K \by M_\lam = K$). 
The embedding $\map{ F }{ K }{ C(K) }$ defined by $F(x) = \seqnn{ f(x) }{f \in C(K) }$ induces a one-to-one map 
$\map{ F_\alpha }{ K/M_\alpha }{ C(K) \cap M_\alpha }$\,. 
Therefore $ \w(K/M_\alpha) \leq | M_\alpha | < \mu$ and thus $K/M_\alpha$ is $\kappa$-Corson, for every $1 \leq \alpha<\lambda$.

Let $K_\al := K \by M_\al$ and by Lemma~\ref{lemma-uiouoi},  let $\map{{\RR}_\al}{K}{K_\al}$ be the associated retraction. We also define  $K_0 \subs K_1$ to be a fixed one-element subset $\{ x_0\}$ of $K_1$ and we define $\map{{\RR}_0}{K}{K_0}$ to be the constant mapping. 
Furthermore, 
\begin{itemize}
\item[{\rm}]
$K_\al \eqdef K/M_\alpha = \cl(K \cap M_\al)$ 
and ${\RR}_\al$ is identity on $K_\al$.
\end{itemize}
Note that ${\RR}_\al \cmp {\RR}_\beta = {\RR}_{\min(\al,\beta)}$ for every $\al,\beta < \lam$. 
So, we have a chain $\sett{K_\al}{\al \loe \lam}$ such that $K_\lam = K$ and $K_\delta = \cl(\bigcup_{\xi < \delta} K_\xi)$ for every limit ordinal $\delta \loe \lam$.
As $K$ has local tightness $<\kappa$ (Proposition~\ref{prop-gsdfjgjhsdf}), we know that $K_\delta = \bigcup_{\xi < \delta} K_\xi$, whenever $\delta$ is a limit ordinal. 
For $\alpha<\beta$,  considering the retraction 
$\fnn{ {\RR}_{\alpha}^{\beta}  \eqdef  \RR_{\alpha} {\restriction} K_\beta\,  }
{\,  K/M_{\beta} }{ K/M_\alpha }$,
we  have the following situation:
$$\begin{tikzcd}
	& & & K \ar[dll]_{{\RR}_\al} \ar[d]^{{\RR}_{\al+1}} \\
	K_0 \cdots & \cdots K\by M_\al \ar[l] & & K\by M_{\al+1} \ar[ll]^{{\RR}_{\al}^{\al+1}} \cdots & \cdots K \ar[l] \, . 
\end{tikzcd}$$

\noindent
We shall need the following claims.

\begin{claim}\label{LMbfskvbdso}
	Assume $K$ is $\kappa$-Corson of weight $<\mu$ and $A \subs K$ is a closed set.
	Then there exists a family $\Yu$ consisting of open $F_\sig$ subsets of $K$ such that $\bigcup \Yu = K \setminus A$, $\Yu$ is $T_0$ separating on $K \setminus A$, and 
	$\ord(x, \Yu)  < \kappa$ for every $x \in K$.
\end{claim}

\begin{pf}
	Let $L := K \by A$, that is, $L$ is obtained from $K$ by collapsing $A$ to a single point $a$. Let $\map q K L$ be the corresponding quotient map. So $q^{-1}(a) = A$ and $q$ is one-to-one on $K \setminus A$.
	Then $L$ has Property $\band{\kappa}$, being a continuous image of $K$ (by Lemma~\ref{LMnioheboier}).
	We already know that $\band{\kappa}$ of $L$ implies that $L$ is $\kappa$-Corson, because the weight of $L$ is $< \mu$. 
	Thus, we may assume that $L \subs \Sigma_\kappa(S)$ for some set $S$.
	It is well-known that $\Sigma_\kappa(S)$ is homogeneous with respect to points, therefore we may assume that $a = 0$.
	Define
	$$\Yu = \setof{q^{-1} \pi_\al^{-1} (r,1]}{\al \in S,\; r \in [0,1] \cap \Qyu}.$$
	Then $\Yu$ is the required family.
\end{pf}

\begin{claim}\label{LMpointwiselimits}
	Assume $x \ne y$ in $K$. Let $\beta \loe \lam$ be the minimal ordinal such that ${\RR}_\beta(x) \ne {\RR}_\beta(y)$. Then $\beta$ is a successor ordinal.
\end{claim}

\begin{pf}
	Note that $\beta>0$, because $K_0$ is a singleton. Suppose $\beta$ is a limit ordinal. Then ${\RR}_\al(x) = {\RR}_\al(y)$ for $\al < \beta$, which means that $x \sim_{M_\al} y$ for $\al < \beta$. On the other hand, $x \not \sim_{M_\beta} y$, therefore there is $f \in C(K) \cap M_\beta$ such that $f(x) \ne f(y)$. We have that $M_{\beta} = \bigcup_{\al < \beta }M_\al$, because $\beta$ is a limit ordinal.
	Thus $f \in M_\al$ for some $\al < \beta$, showing that $x \not \sim_{M_\al} y$, a contradiction.
\end{pf}

We now come back to the proof of Theorem~\ref{thm-753}:  
bear in mind that $\lam = \cf \mu$ where $\w(K) = \mu \goe \kappa$, and the theorem is valid for compact spaces of weight $< \mu$. 
We have seen that $\w(K_\alpha) \leq | M_\alpha | < \mu$. 
So for each $\alpha<\lambda$, since $K_\alpha$ satisfies $\band{\kappa}$, the space $K_\alpha$ is $\kappa$-Corson and thus using Claim~\ref{LMbfskvbdso} we choose a family $\Vee_\al$ of open $F_\sig$ subsets of $K_{\al+1}$ such that $\bigcup \Vee_\al = K_{\al+1} \setminus K_\al$ and $\ord(x, \Vee_\al) < \kappa$ for every $x \in K_{\al+1}$.
Define
$$\Yu = \setof{{\RR}_{\xi+1}^{-1} V}{ \xi < \lam, \; V \in \Vee_\xi}.$$
We claim that $\Yu$ is a family witnessing that $K$ is $\kappa$-Corson.

Clearly, $\Yu$ consists of open $F_\sig$ sets. We show that $\Yu$ is $T_0$ separating.

Fix $x \ne y$ in $K$ and let $\beta$ be minimal such that ${\RR}_\beta(x) \ne {\RR}_\beta(y)$.
Then $\beta = \al+1$ for some $\al<\lam$, by Claim~\ref{LMpointwiselimits}.

If ${\RR}_\beta(x) \notin K_\al$ or ${\RR}_\beta(y) \notin K_\al$ then there is $V \in \Vee_\al$ separating ${\RR}_\beta(x)$ from ${\RR}_\beta(y)$ and hence ${\RR}_\beta^{-1}V \in \Yu$ separates $x$ from $y$.
Suppose now that ${\RR}_\beta(x) \in K_\al$ and ${\RR}_\beta(y) \in K_\al$.
Then ${\RR}_\al(x) = {\RR}_\al {\RR}_\beta(x) = {\RR}_\beta(x)$ and similarly ${\RR}_\al(y) = {\RR}_\beta(y)$. On the other hand, by the minimality of $\beta$ we get ${\RR}_\al(x) = {\RR}_\al(y)$, a contradiction.

It remains to show that $\ord(x,\Yu) < \kappa$ for every $x \in K$.
Suppose $\sett{{\RR}_{\al_i+1}^{-1}V_i}{i < \kappa} \subs \Yu$ has nonempty intersection. Refining this family, we may assume that $\sett{\al_i}{i < \kappa}$ is either strictly increasing or constant.
Choose $a \in \bigcap_{i<\kappa}{\RR}_{\al_i+1}^{-1}V_i$  ($\subseteq K$).
If $\al_i = \al$ for $i < \kappa$, then $\sett{V_i}{i < \kappa} \subs \Vee_\al$ and hence ${\RR}_{\al+1}(a) \in \bigcap_{i<\kappa}V_i$, contradicting the fact that $\ord(x, \Vee_\al) < \kappa$ for $x \in K_{\alpha+1} \subseteq K$.
Thus, we may assume that $\sett{\al_i}{i < \kappa}$ is strictly increasing. Let $\delta = \sup_{i < \kappa}\al_i$.
Then $\delta \loe \lam$ is a limit ordinal of cofinality $\kappa$.
Let $b = {\RR}_\delta(a)$. Then $b \in K_\delta$ and ${\RR}_{\al_i+1}(b) = {\RR}_{\al_i+1}(a) \in V_i$ for $i<\kappa$. 
On the other hand, $K_\delta = \bigcup_{\xi < \delta}K_\xi$ (because the local tightness of $K$ is $<\kappa$), therefore $b \in K_\al$ for some $\al < \delta$. 
Now fix $\alpha_i > \alpha$. Then $b \notin V_i$, because $V_i \cap K_{\al_i} = \emptyset$. This is a contradiction.

Thus we have shown that $\ord(x, \Yu) < \kappa$ for every $x \in K$, which completes the proof of Theorem~\ref{thm-753}.	
\end{pf} 

\subsection{Basic properties of $\kappa$-Corson spaces}
\label{properties of Corson-like-space}

As a consequence of Lemma~\ref{LMnioheboier} and of Theorem~\ref{thm-753} we have: 

\begin{wn}\label{WNsduhsd}
\label{cor-hfdjskhjkf}
Assume $\kappa$ is an uncountable regular cardinal. Every continuous image of a $\kappa$-Corson compact space is $\kappa$-Corson. 
\qed 
\end{wn}

In case $\kappa = \aleph_1$, namely, for Corson compacta, the result above was first proved by Gul$'$ko~\cite{Gulko} (see also \cite{MiRu} and \cite{Gru}). We have obtained it by means of elementary submodels, however, extracting the necessary closing-off arguments one can possibly obtain a slightly more direct proof of Corollary~\ref{WNsduhsd}.

Concerning further basic properties of $\kappa$-Corson spaces, we state three results.

\begin{tw}
\label{Thmweteig7878}
\label{thm-3.57878}
Let $\kappa = \cf \kappa > \aleph_0$ and assume that $K$ is a continuous image of a $\kappa$-Valdivia compact space. 
Then $K$ is $\kappa$-Corson 
if and only if $K$ has local tightness $< \kappa$.
\end{tw}

\begin{proof}
Suppose first that $K$ is $\kappa$-Corson. 
Then, by Theorem~\ref{prop-gsdfjgjhsdf}, $K$ has Property $\band \kappa$ and thus $K$ has local  tightness $< \kappa$. 

Now suppose $\map f L K$ is a continuous surjection with $L$ $\kappa$-Valdivia and assume $K$ has local tightness $< \kappa$. We may assume $L \subs [0,1]^\lam$ so that $D := L \cap \Sigma_\kappa(\lam)$ is dense in $L$.
Note that $K = \img f D$, because of the local tightness $<\kappa$. Indeed, the set $\img f D$ is dense and $\kappa$-closed in $K$, namely, the closure of every subset $B \subs \img fD$ of cardinality $<\kappa$ is contained in $\img fD$.

We now show that $K$ has Property $\band{\kappa}$.
Fix a suitable $M \rloe H(\theta)$ with $|M| < \kappa$ and $f \in M$.
Fix $y_0 \ne y_1$ in $\cl(K \cap M)$.
Let $S := \lam \cap M$ and let
$$L_S := \setof{x \in L}{\supp(x) \subs S}.$$
Let $A_i := f^{-1}(y_i) \cap L_S$. Then $A_0$, $A_1$ are disjoint closed subsets of $L_S$.
By compactness and normality, there are disjoint open sets $U_0 \sups A_0$, $U_1 \sups A_1$ that are unions of basic open sets. In other words, $U_0$, $U_1$ are determined by disjunctions of formulas of the form $x(\al_0) \in I_0 \land \dots \land x(\al_{k-1}) \in I_{k-1}$, where $x$ is a variable indicating an element of $L$ and $I_i$, $i<k$ are open intervals with rational end-points. In particular, $U_0, U_1 \in M$. We now look at these sets in $L$ (using the same definitions, namely, formulas involving finitely many coordinates and associated rational intervals). If there is $x \in U_0 \cap U_1$, then $x \cdot \chi_S$ satisfies the same requirements, namely, $x \cdot \chi_S \in U_0 \cap U_1$. We will obtain a contradiction as long as we show that $x \cdot \chi_S \in L_S$.

Suppose otherwise and choose a basic neighborhood $W$ of $x \cdot \chi_S$ disjoint from $L$, involving finitely many coordinates and open intervals with rational end-points. Let $W_S$ be the restriction of $W$ obtained by discarding all the coordinates outside of $S$. So $W_S \sups W$ and $W_S \in M$. Now $M$ knows that $L\cap W_S \nnempty$, because $x \in L \cap W_S$. By elementarity, there is $x' \in L \cap M$ with $x' \in W_S$. Note that $x' \in W$, because $x'(\al) = 0 = x(\al)$ whenever $\al \in \lam \setminus S$.
Thus $x' \in W \cap L$, a contradiction.

We have shown that $U_0$, $U_1$ are disjoint open sets in $L$ that are also elements of $M$. Finally, $\img f {L \setminus U_0}$, $\img f {L \setminus U_1$} is a closed cover of $K$ known by $M$ and separating $y_0$ from $y_1$.

Indeed, suppose $y_0 \in \img f {L \setminus U_0}$. Then there is $x_0 \in L$ with $f(x_0) = y_0$ and $x_0 \notin U_0$. By the arguments above, $x_0 \cdot \chi_S \in L$ and it is outside of $U_0$, as the definition of $U_0$ uses coordinates from $S$ only. On the other hand, every element of $A_0 = f^{-1}(y_0) \cap L_S$ is in $U_0$, a contradiction.
Thus $\img f {L \setminus U_0}$, $\img f {L \setminus U_1$} indeed separate $y_0$, $y_1$.
This completes the proof.
\end{proof} 

Let $X$ be a space, $x \in X$  and let $\kappa$ be an infinite  cardinal. 
We say that $x$ is a \dfn{$G_\kappa$-point}{$G_\kappa$-point} 
if $\{x\}$ is the intersection of less than $\kappa$ open sets. 
So a $G_{\aleph_1}$-point is a $G_{\delta}$-point. 
\koment{Robert comment: Say that $x$ is a strong $G_\kappa$-point if $x$ is a $G_\kappa$-point but not a $G_\lambda$-point for $\lambda<\kappa$. Does this notion has some interest?}

\begin{proposition}
\label{fact-uiuio}
Assume $\kappa$ is an infinite\footnote{It was pointed out by the Referee that the proof of Proposition~\ref{fact-uiuio} is valid without the assumption on the regularity of $\kappa$.} cardinal.
Then any $\kappa$-Corson space has a $G_\kappa$-point. 
Moreover, the set of $G_\kappa$-points is topologically dense in the space.
\end{proposition}

\begin{proof}
 We adapt the proof to the well-known case where $K$ is a Corson compact space and $G_\lambda = G_\delta$. 
Let $S$ be such that $K$ is a compact subspace of  $\Sigma_\kappa(S) \subseteq [0,1]^S$.
For $x, y \in K$, we set $x \preceq y$ whenever $y$ extends $x$, i.e. $\supp(x) \subseteq \supp(y)$ and
$x(\alpha)=y(\alpha)$ for every $\alpha \in \supp(x)$.  
Let $C$ be a maximal chain in $\pair{K}{\preceq}$ and 
let $z = \bigcup C$. 
Since $K$ is compact and $z \in [0,1]^S$ is an accumulation point of $C$, we have that $z \in K$ and thus $ | \supp(z) | < \kappa$.

We prove that $z$ is a $G_\kappa$-point in $K$. 
Let $V$ be a basic open neighborhood of $z$ in $\Sigma_\kappa(S)$. 
So $V$ is determined by a finite subset $\sigma$ of $\supp(z)$: that is $V = \prod_{\beta \in \supp(z)} V_\beta$ where $ V_\beta = [0,1]$ for $\beta \in S \setminus \sigma$ and 
$V_{\beta} \eqdef (x(\beta) {-}r, \, x(\beta) {+}r) \cap [0,1]$ for $\beta \in \sigma$ and $r \in \natQ$. 
From the facts that 
(1):~the set of finite subsets $\sigma$ of $\supp(z)$ is of cardinality less than $\kappa$, 
(2):~$ | \supp(z) | < \kappa$ and thus the set of basic open neighborhoods $V$ of $z$ of the above form is of cardinality less than $\kappa$ and 
(3):~by the maximality of $C$: 
if $t \in K$ satisfies $t(\alpha)=z(\alpha)$ for every $\alpha \in S$ then $t = z$, it follows that $z$ is a $G_\kappa$-point. 

The topological density of $G_\kappa$-points follows from the compactness of $K$ and from Proposition~\ref{prop-yfsdio}(3).
\end{proof}

 Recall that for distinct infinite regular cardinals $\lambda$ and $\kappa$ we can have $2^{\lambda} = 2^{\kappa}$ and thus 
 $ \bigl| \, [0,1]^{\lambda} \,\bigr| = \bigl| \, [0,1]^{\kappa}\, \bigr| $. 
But $[0,1]^{\lambda}$ and $[0,1]^{\kappa}$ are not homeomorphic. 
This is an analogous result as Banach-Stone theorem.

\begin{corollary}\label{cor-ieyzuirzeuirez}
Let $\lambda$ and $\kappa$ be infinite cardinals. 
\begin{itemize}
\item[{\rm(1)}]
The space $[0,1]^{\lambda}$ is $\kappa$-Corson if and only if $\lambda < \kappa$.
\item[{\rm(2)}]
The space $\{ 0,1\}^{\lambda}$ is $\kappa$-Corson if and only if $\lambda < \kappa$. 
\end{itemize}
Moreover, the space $\kappa+1$ is $\kappa$-Valdivia but not $\kappa$-Corson.
\end{corollary}
\begin{proof} 
Since the product of $<\kappa$ \ $\kappa$-Corson spaces is $\kappa$-Corson, if $\lambda<\kappa$ then $[0,1]^{\lambda}$ is $\kappa$-Corson. 
By Proposition~\ref{fact-uiuio}, $\{0,1\}^\kappa$ cannot be $\kappa$-Corson, as it clearly has no $G_\kappa$-points. Hence, $[0,1]^\kappa$ is not $\kappa$-Corson, being a superspace of $\{0,1\}^\kappa$.

For the ``Moreover'' part, notice that $\kappa+1$ is homeomorphic to the space of all non-increasing functions from $\kappa$ to $\{0,1\}$, hence it is $\kappa$-Valdivia. It is not $\kappa$-Corson, because the maximal element has tightness $\kappa$ in $\kappa+1$.
\end{proof}
A recent note of Plebanek~\cite{Ple} contains a result showing that neither $[0,1]^\kappa$ nor $\{0,1\}^\kappa$ can be $\kappa$-Corson for an arbitrary (possibly singular) infinite cardinal $\kappa$, answering a question of Kalenda~\cite{Ka1}.

\subsection{The Lindel\"of number of spaces of continuous functions}
\label{begin-lindelof-preservation}

Let $\kappa$ be an infinite cardinal and let $E$ be a topological space. 
We say that $E$ is \dfn{$\kappa$-Lindel\"of}{$\kappa$-Lindel\"of} if every open cover of $E$ contains a subcover of cardinality $<\kappa$. 
So the case $\kappa=\aleph_1$ corresponds to the classical notion of  Lindel\"of space. 

Assume that $K$ is $\kappa$-Corson.
Recall that by Theorem~\ref{thm-753} and Lemma~\ref{lemma-uiouoi}, the retraction ${\RR}_M$ maps $K/M$ onto $\cl(K \cap M)$ and that ${\RR}_M$ is a homeomorphism onto.
So ${\RR}_M$  induces a continuous embedding ${\RR}_M^*$ from $C(\cl(K \cap M))$ into $C(K)$.
Note that for a general compact space $K$, we have a canonical embedding of $C(K \by M)$ into $C(K)$, induced by the quotient mapping $\map{q_M}{K}{K \by M}$.

The proof of the following result is a suitable adaptation of the proof of~\cite[\S17.3, Theorem 17.1]{KKLP}.

\begin{tw}   
\label{thm-456789} 
Let $\kappa$ be an uncountable regular cardinal and let $K$ be a $\kappa$-Corson compact space.
Then $C(K)$, endowed with the pointwise topology, is $\kappa$-Lindel\"of. 
\end{tw}

\begin{pf} 
	Recall that the canonical open base of $\pair{C(K)}{\tau_p}$ consists of sets of the form
	$$V = \setof{f \in C(K)}{(\forall\; i < n) \; f(x_i) \in J_i},$$
	where $\ntr$, $x_0, \dots, x_{n-1} \in K$ and $J_0, \dots, J_{n-1}$ are open intervals with rational end-points.
	Let $\Vee$ be a cover of $C(K)$ consisting of sets of the above form.
	Fix a $\kappa$-stable elementary submodel $M$ of a big enough $H(\theta)$ such that $|M| < \kappa$ and $\Vee \in M$ (so, in particular, $K \in M$). 
	It suffices to show that $\Vee \cap M$ covers $C(K)$. 
	
	By Theorem~\ref{thm-753}, $K$ has property $\band{\kappa}$. 
Namely, let $\map {{\RR}_{M}}  K K$ be the retraction induced by $M$. Recall that $\img {{\RR}_{M}}  K = \cl(K \cap M)$. Let $\map{\pi}{C(K)}{C(K)}$ be the dual projection, namely, 
$\pi : f \mapsto f \circ {\RR}_{M}$, i.e. $(\pi f)(x) = f({\RR}_{M} (x))$ for $f \in C(K)$. 	
	
	Note that the image of $\pi $ is $\cl(C(K) \cap M)$, where the closure is in the norm topology. 
	
	Fix $g \in C(K)$.
	Choose a rational $\eps>0$ such that the ball $B(\pi g, \varepsilon)$ is contained in some element of $\Vee$. 
	Since $\pi[C(K)] =\cl(C(K) \cap M)$, choose 
	$h \in C(K) \cap M$ with $\norm{\pi g - h} < \eps/2$.
	Let $B = B(h,\eps/2)$. 
	Since $h, \varepsilon \in M$ we have $B \in M$. 
	 Moreover $B \subs B(\pi g,\eps)$ therefore $B$ is contained in some element of $\Vee$.
	By elementarity, there is $V \in \Vee \cap M$ such that $B \subs V$.
	Assume that 
\smallskip
\newline
\smallskip
\centerline{	
	$V = \setof{f \in C(K)}{(\forall\; i < n) \; ( f(x_i) \in J_i ) } $,
} where $x_0, \dots, x_{n-1} \in K$ and $J_0, \dots, J_{n-1}$ are rational intervals. Note that 
\smallskip
\newline
\smallskip
\centerline{	
$x_0, \dots, x_{n-1} \in M,$
} 
again by elementarity, because
$V \in M$ and $M$ knows that it is a basic set.
	Moreover $\pi g \in B \subseteq V$, which means that $\pi g(x_i) \in J_i$ for $i<n$.
	
	Now note that, given $x \in \cl(K \cap M)$ we have that $(\pi f)(x) = f({\RR}_{M} (x))= f(x)$ for every $f \in C(K)$.
	Finally, by the remark above, $g \in V$, because $g(x_i) = \pi g(x_i)$ for $i<n$.
	We have proved that $\Vee \cap M$ covers $C(K)$. 
\end{pf}

\subsection{Examples of $\kappa$-Corson spaces}
\label{Corson-like-space}

Let $K$ be a compact space and let $D \subs K$.
We define the Alexandrov duplication $\alex(K,D)$ as follows. 
We assume that $(K \times \{0\}) \cap (D \times \{1\}) = \emptyset$. 
We set $\alex(K,D) = (K \times \{0\}) \cup (D \times \{1\})$ 
endowed with the topology for which 
a point of the form $\pair{x}{1}$, with $x \in D$, is isolated 
and a neighborhood of a point of the form $\pair{x}{0}$, with $x \in K$, is of the form 
$\bigl( (V \times 2) \cap \alex(K,D) \bigr) \setminus \{ \pair{x}{1} \}$ 
where $V$ is a neighborhood of $x$ in $K$.

The following result is rather standard. 

\begin{prop}
\label{Prowngog}
Let $K$ be a $\kappa$-Corson compact space and let $D \subs K$.
Then the Alexandrov duplication $\alex(K,D)$ is $\kappa$-Corson.
\end{prop}

\begin{proof}
Assume that $K$ is $\kappa$-Corson. 
It is easy to check that $\alex(K,D)$ is compact.
Let $\Yu$ be family of open $F_\sig$ subsets of $K$ as in Proposition~\ref{Propeholnjewr}(ii). 
We define a family $\Vee$ of subsets of $\alex(K,D)$ as follows. 
$$\Vee = \setof{ \, \{ \pair{x}{1} \} \,  }{  \,   x \in D \,  } \cup \, 
 \setof{  \,  (V \times 2) \cap \alex(K,D)  \,  }{  \,  V \in \Yu  \,  } \, . $$
We leave to the reader to verify that this set $\Vee$ consists 
of open $F_\sig$ subsets of $\alex(K,D)$ and that $\Vee$
satisfies Proposition~\ref{Propeholnjewr}(ii). 
\end{proof}

Bishop and de Leeuw \cite[\S VII: Examples]{BDL} generalize the above Alexandrov duplication $\alex(K,D)$ in the following way. 
The results in \cite{BDL} are related to Choquet Theorem.

Let $\pair{ Y_x }{ \tau_x }_{ x \in X }$ be a family of pairwise disjoint nonempty topological spaces indexed by a space 
$\pair{ X }{ \tau_X }$, \,$Y \eqdef \bigcup_{x \in X } Y_x$ and 
let $\map{ s }{ X }{ Y }$  be such that $s(x) \in Y_x$ for every $x \in X$. 
We define the \dfn{porcupine space}{porcupine space} 
$\vec{\calY} \eqdef \trpl{ Y }{ \tau }{ s }$.
The universe of $\vec{\calY}$ is $Y = \bigcup_{x \in X } Y_x$. 
We denote by $\map{ \pi }{ Y }{ X }$ the projection, that is $\pi(y) = x$ if and only if $y \in Y_x$. 
So $\pi s (x) = x$ for $x \in X$.
We describe the porcupine topology $\tau$ on $Y$ as follows. 
Let
\begin{itemize}
\item[{\scriptsize\rm$\bullet$}]
$\tau^0$ be the set of all subsets $U$ of $Y$   satisfying: 
there is $x \in X$ such that 
$U \in \tau_x$ and $s(x) \not\in U$, and 
\item[{\scriptsize\rm$\bullet$}] 
$\tau^1$ be the set of all subsets of $Y$ of the form  $\pi^{-1}[V]$ or of the form 
\\
\centerline{
$\pi^{-1}[V {\setminus} \{x\}] \cup U
\eqdef \bigcup \, \setof{ Y_t }{ t \in V \text{ and } t \neq x } \cup U$ 
} 
where 
$V \in \tau_X$\,, \ $x \in V$, and $U \in \tau_x$ satisfies 
$s(x) \in U$. 
\end{itemize}
In other words a member of $\tau^1$ has the form:
$\pi^{-1}[V] \setminus F
\eqdef \bigcup \, \setof{ Y_t }{ t \in V } \setminus F$  
where 
$V \in \tau_X$, and for some $x \in X$ the set $F$ is a closed subset of $Y_x$ satisfying $s(x) \not\in F$. 

The topology $\tau$ on $Y$ having $\tau^0 \cup \tau^1$ as a subbase is called the \dfn{porcupine topology}{porcupine topology} $\tau$ on $Y$. 
Note that if $X$ and each $Y_x$ are $T_0$ (resp. $T_2$) and compact, then $Y$ is $T_0$ (resp. $T_2$) and compact.

\begin{Observation*}
\label{observation-1} 
Assume that $X$ and each $Y_x$ is $T_2$ (and thus $Y$ is $T_2$). 
\begin{itemize}
\item[{\rm(1)}]
The set $Y_x \setminus \{ s(x) \}$ is open in the subspace $Y_x$ and in the subspace $Y \setminus \{ s(x) \}$ for every $x \in X$.
\\ 
Moreover the induced topologies of $\pair{Y_x}{\tau_x}$  and of  $\pair{Y}{\tau}$ on the set $Y_x \setminus \{ s(x) \}$ coincide.
\item[{\rm(2)}] 
The set $s[X]$ is closed in the space $Y$. 
\\ 
Moreover $s$ is a homeomorphism from $X$ onto the subspace $s[X]$ of $Y$. 
\end{itemize}
\end{Observation*}

\begin{proof}[Hint]
(1) uses the fact that $Y_x \setminus \{ s(x) \} \in \tau^0$. 

(2) uses the fact that, by Part~(1), $Y \setminus s[X] = \bigcup_x (Y_x  \setminus \{s(x)\})$ is open.  
\end{proof}

\begin{prop}
\label{Prowngog-porcupine} 
Let $\kappa$ be a regular cardinal, 
$\pair{ Y_x }{ \tau_x }_{ x \in X }$ be a family of 
pairwise disjoint $\kappa$-Corson spaces indexed 
by a $\kappa$-Corson space $\pair{ X }{ \tau_X }$ 
and let $\map{ s }{ X }{ Y }$  be such that 
$s(x) \in Y_x$ for every $x \in X$. 
We suppose that the following hold.
\begin{itemize}
\item[{\rm(1)}]
For every $x \in X$, $\{ s(x) \}$ is $G_\delta$ in $Y_x$\,. 
\item[{\rm(2)}]
The set 
$\Sigma_{\vec{\calY}} \eqdef 
\setof{ \, x \in X }{ \text{\rm there is } U \in \calU_x 
\text{\rm\ such that } U \neq Y_x 
\text{\rm\ and } s(x) \in U  \, }$
is of cardinality less than $\kappa$, 
where for every $x \in X$ the family $\calU_x$ of $F_\sigma$ open sets  in $Y_x$ satisfies Proposition~{\rm\ref{Propeholnjewr}(ii)}.
\end{itemize}
Then the porcupine space $\vec{\calY} \eqdef \trpl{ Y }{ \tau }{ s }$ is $\kappa$-Corson. 
\end{prop}

\begin{proof}
Choose a family $\calU_X$ of $F_\sigma$ open sets in $X$ 
satisfying Proposition~\ref{Propeholnjewr}(ii). 
Recall that $\calU_x$ witnesses the fact that $Y_x$ is $\kappa$-Skula (i.e. $\calU_x$ satisfies Proposition~\ref{Propeholnjewr}(ii)) for $Y_x$.

For every $x \in X$, as above, we denote by 
\begin{itemize}
\item[{\scriptsize\rm$\bullet$}]
$\calV_x^0$ the set of all subsets $V$ of $Y$ such that $V \in \calU_x$ and $s(x) \not\in V$, and by
\item[{\scriptsize\rm$\bullet$}] 
$\calV_x^1$ the set of all subsets of $Y$ of the form 
$\pi^{-1}[W \setminus \{x\}] \cup U$, where 
$W \in \calU_X$\,, \ $x \in W$  and $U \in \calU_x$ satisfies $s(x) \in U$. 
\end{itemize}
We set $\calV^0 = \bigcup_{x \in X} \calV_x^0$, 
\,$\calV^1 = \bigcup_{x \in X} \calV_x^1$ and  
$\calV = \calV^0 \cup \calV^1$. 
Also for $x \in X$ we set $\calV_x = \calV_x^0 \cup \calV_x^1$. 

We prove that $\calV$ satisfies Proposition~\ref{Propeholnjewr}(ii). 
First, using if necessary the fact that $\{ s(x) \}$ is a $G_\delta$ in $Y_x$,  each $V \in \calV$ is an $F_\sig$ in~$Y$. 
Then $\calV$ is $T_0$-separating. 
In order to see this, fix $y \ne y'$ in $Y$. 
If $y, y' \in Y_x \setminus \{ s(x) \}$ then $y$ and $y'$ are separated by a member of $\calV_x^0$.
If $y = s(z)$ and $y' = s(z')$ then $y$ and $y'$ are separated by some $\pi^{-1}[W] \in \calV^1$. 
If $y \in Y_x$ and $y' = s(x)$ then $y$ and $y'$ are separated by some  $\pi^{-1}[W \setminus \{ x \}] \cup U  \in \calV^1$. 

For $y \in Y$, and $\calW \subseteq \calV$, we set $\calW(y) \eqdef \setof{ V \in \calW }{ y \in V }$,  
and thus $\ord(y,\calW)  = | \calW(y) |$.
Since $Y$ is $T_0$ and compact,  
to see that $Y$ is $\kappa$-Corson, it remains to show that 
$ | \calV(y) | < \kappa $ for every $y \in Y$. 
We fix $y \in Y$. 
So $y \in Y_{x(y)}$ for an unique $x(y) \in X$. 
Since $\calV^0(y) = \calV_{x(y)}^0(y)$, we have 
$ | \calV^0(y) | < \kappa$. 

To prove that $ | \calV^1(y) | < \kappa$ we need some notions. 
Any $V \in \calV^1$ has the form 
$$V_{z,W,U} \eqdef \pi^{-1}[W \setminus \{z\}] \cup U,$$
where 
$W \in \calU_X$\,, \ $z \in W$  and $U \in \calU_z$ satisfies $s(z) \in U$. 
By the hypotheses, the sets 
$\setof{ \pi^{-1}[W] }{ z \in W  \in \calU_X }$ and 
$\setof{ U \in \calU_z }{ s(z) \in U \in \tau_x }$ are of cardinality $<\kappa$ and thus   
$$\calV_{z} \eqdef \setof{ V_{z,W,U} \in \calV^1 }{ z \in W \in \calU_X \text{ and } s(z) \in U \in \tau_{z} },$$
is of cardinality $< \kappa$, for any $z \in X$. In particular 
\begin{itemize}
\item[{\rm(i)}]
$\calV_{z}(y) \eqdef \setof{ V_{z,W,U} \in \calV_z }{ y \in V_{z,W,U} }$ 
is of cardinality $< \kappa$ for any $z \in X$ and $y \in Y$. 
\end{itemize}
Recall that $y \in Y_x$. 
\begin{itemize}
\item[{\rm(ii)}]
If $y = s(x)$ then, by~(i), $ \calV^- \eqdef \calV_{x}(y)$
is of cardinality $< \kappa$.
\end{itemize}
Now suppose that $ y \neq s(x) $. 
If $z \in X \setminus \Sigma_{\vec{\calY}}$, 
then $U = Y_z$ and thus $s(z) \in V_{z,W,U} = \pi^{-1}[W \setminus \{z\}] \cup Y_z = \pi^{-1}[W]$. 
Therefore 
\begin{itemize}
\item[{\rm(iii)}]
$\calV^{*} \eqdef \bigcup \setof{  \calV_{z}(y) }
{z \in X {\setminus} \Sigma_{\vec{\calY}} }
= \setof{ \pi^{-1}[W] }{ W \in \calU_X \text{ and } s(y) \in W }$
is of cardinality $< \kappa$ for every $y \neq s(x)$.
\end{itemize}
Since $ | \Sigma_{\vec{\calY}} | < \kappa $, and again by~(ii) for every $z \in \Sigma_{\vec{\calY}}$ we have $ | \calV_z(y) | < \kappa$, and thus 
\begin{itemize}
\item[{\rm(iv)}]
$\calV^{**} \eqdef \bigcup \setof{  \calV_{z}(y) }
{z \in \Sigma_{\vec{\calY}} }$
is of cardinality $< \kappa$  for every $y \neq s(x)$.
\end{itemize}
Since $\calV^1(y) = \calV^- \cup \calV^{*} \cup \calV^{**}$, 
by~(ii)-(iv), $ | \calV^1(y) | < \kappa$ for any $y \in Y$.
\end{proof}

The next result is a re-statement of a result of Bonnet and Rubin in \cite{BR} \S 3.3 (Proposition 3.16).

\begin{lm}
\label{lemma-789456}
Let $\vec{\calY} \eqdef \trpl{ Y }{ \tau }{ s }$ be a porcupine family such that $X$ and each $Y_x$ are compact and $0$-dimensional.
Then $Y$ is compact and $0$-dimensional. 

Moreover if $X$ and each $Y_x$ is scattered then $Y$ is a scattered space.
\qed
\end{lm}

\section{A characterization of $\kappa$-Corson Boolean algebras}
\label{Corson-like-BA}

We follow Koppelberg~\cite{Ko} for Boolean algebraic notations. 
In a Boolean algebra 
$\Be$, $\, +$ and $\cdot$   
denote the join and the meet, and ${-}x$ is the complement of $x$ in $B$. 
Also for $x,y \in \natB$, $x \cdot -y$ is denoted by $x-y$, and  $x  \symdiff  y$ 
denotes the symmetric difference, that is $(x-y)+(y-x)$.
For a Boolean algebra $\natB$ we denote by \dfn{$\ult(\natB)$}{$\ult(\natB)$: set of ultrafilters} the set of ultrafilters of $\natB$ (with the pointwise topology). 
Also if $X$ is a 0-dimensional compact space, we denote by \dfn{$\clop(X)$}{$\clop(X)$ set of clopen sets} 
the set of clopen subsets of $X$, which is a Boolean algebra.
For a set $X$, we denote by \dfn{$\Fr(X)$}{$\Fr(X)$: free Boolean algebra} the \dfn{free Boolean algebra}{free Boolean algebra} over $X$. 
For interval algebras, we refer to Koppelberg~\cite[\S15]{Ko}.
For poset Boolean algebras, we refer to Abraham, Bonnet, Kubi\'s and Rubin \cite{ABKR}.

In a $0$-dimensional compact space, for a family of clopen sets: the property of being $T_0$ separating is equivalent to being a generating set in the algebra of all clopen sets. 
So another formulation of Proposition \ref{Propeholnjewr} is:

\begin{wn}
\label{cor-generation}
Let $\kappa > \aleph_0$ be a cardinal and $K$ be a $0$-dimensional compact space. 
Denote by $\Be$ the Boolean algebra of clopen subsets of $K$.
The following are equivalent.

\begin{enumerate}
\item[{\rm(i)}]
The space $K$ is $\kappa$-Corson.
\item[{\rm(ii)}] 
The Boolean algebra $\Be$  is $\kappa$-Corson.
\qed
\end{enumerate}
\end{wn}

Let $K$ be a topological space and $A \subseteq K$.
We denote by $\cl(A)$ the topological closure of $A$ in $K$.
For $A \subseteq \ult( \Be)$, recall that $p \in \cl (A) \subs \ult( \Be)$ if and only if for every $a \in p$ there is $x\in A$ with $a \in x$.
In other words, for a set $A \subs \ult (\Be)$, we have 
\begin{equation}
\cl (A) = \setof{\, p\in \ult (\Be)}{p \subs \bigcup A \,}.
\tag{$\star$}\label{star}
\end{equation}
We shall use this formula several times later.

Now fix a Boolean algebra $\Be$ and fix a big enough regular cardinal $\theta$ so that $\Be \in H(\theta)$.
Furthermore, fix an elementary submodel $M$ of $\pair {H(\theta)}\in$ such that $\Be \in M$.
It is obvious that $\Be \cap M$ is a subalgebra of $\Be$.
It turns out that in case $\Be$ is $\kappa$-Corson and $M$ is sufficiently closed, the embedding $\Be \cap M \subs \Be$ has a special property. 

Namely, let $K = \ult (\Be)$. 
Then the relation 
$p \subs \bigcup (K \cap M)$
is equivalent to $p \in \cl (K\cap M)$.
Hence the natural equivalence relation $\sim_M$ 
on $\cl (K\cap M) \subseteq K$ says that 
$p \sim_M q$ if and only if there is no $a \in \Be \cap M$ such that $a \in (p \setminus q) \cup (q \setminus p)$.

\begin{prop}
\label{Pnowdela}
Let $\kappa$ be a regular cardinal, let $\Be$ be a 
$\kappa$-Corson Boolean algebra, and let $M$ be an elementary $\kappa$-stable submodel of a big enough $H(\theta)$ such that $\Be \in M$. 
Then for every $p, q \in \ult (\Be)$ such that 
$$p , q \subs \bigcup \, \bigl( \, \ult (\Be) \cap M \, \bigr)$$ 
if $p \cap M = q \cap M$ then $p = q$. 
\qed
\end{prop}

\begin{pf} 
Let $G$ be a generating set witnessing that $\Be$ is $\kappa$-Corson.
By elementarity, we may assume that $G \in M$. 

Suppose $p \ne q$ and 
$p\cup q \subs \bigcup \bigl(\ult (\Be) \cap M \bigr)$, 
that is  $p, q \in \cl (\ult (\Be) \cap M)$. 
Then $p\cap G \ne q\cap G$, because $G$ generates $\Be$.
Choose $a \in p \cap G \setminus q$ (if necessary, we switch $p$ and $q$).
By the assumption, 
$a \in r$ for some $r \in \ult (\Be) \cap M$.
Recall that $|r \cap G| = \lam$ for some $\lam < \kappa$ and by elementarity and the facts that $r, G \in M$,  we conclude that  $\lambda \in \kappa \cap M$ and thus, by $\kappa$-stability, $\lam + 1 \subs M$.
Thus, $r \cap G \subs M$, because $r, \lam, G \in M$.
It follows that $a\in M$ and hence $p \cap M \ne q \cap M$.
\end{pf}

Let $\Be$ be a Boolean algebra and let $M$ be an elementary $\kappa$-stable submodel of some $H(\theta)$.
Motivated by Proposition~\ref{Pnowdela}, we shall say that $M$ is \dfn{$\Be$-good}{$\Be$} if for every $p,q \in \ult(\Be)$ such that $p \ne q$ and $p \cup q \subs \bigcup (\ult(\Be) \cap M)$, it holds that $p \cap M \ne q \cap M$.

\medskip

We shall say that a Boolean algebra $\Be$ has \dfn{property $\band \kappa$}{property $\band \kappa$} if every elementary $\kappa$-stable submodel $M$ of a big enough $H(\theta)$ satisfying $\Be \in M$ and $|M| <\kappa$ is $\Be$-good. 

\medskip

It is rather clear that this property translates to the already defined property $\band{\kappa}$ for compact spaces, namely:

\begin{prop}\label{PropJgweogn}
Let $\kappa$ be an uncountable regular cardinal.
	A Boolean algebra has property $\band{\kappa}$ if and only if its Stone space has property $\band{\kappa}$. 
\end{prop}

\begin{proof}
	It is enough to note that $M$ is $\Be$-good if and only if the equivalence relation $\sim_M$ is one-to-one on $\cl(K \cap M)$, where $K$ is the Stone space of $\Be$, taking into account that $x \not \sim_M y$ iff there is a clopen set in $K \cap M$ separating $x$ and $y$ (see Lemma~\ref{Lmeruvggh}). 
\end{proof}

Theorem~\ref{thm-753} combined with Proposition~\ref{PropJgweogn} gives:

\begin{tw}\label{Thmonbandlow}
Let $\kappa$ be an uncountable regular cardinal.
A Boolean algebra is $\kappa$-Corson if and only if it has Property $\band{\kappa}$.
\qed
\end{tw}

As a consequence of Corollary~\ref{cor-generation} and of the above result~\ref{Thmonbandlow} we have.

\begin{tw}\label{Thmjgwrgo}
Let $\kappa$ be an uncountable regular cardinal and assume $\Be$ is a $\kappa$-Corson Boolean algebra.
Then every subalgebra of $\Be$ is $\kappa$-Corson.
\qed
\end{tw}

\subsection{Remarks on Valdivia algebras}
\label{Valdivia}

We have seen already that a $\kappa$-Corson Boolean algebra has local tightness $<\kappa$, whenever $\kappa$ is a regular cardinal (Proposition~\ref{prop-gsdfjgjhsdf}).
This actually can be proved directly, showing that $\tight(p, A) < \kappa$ whenever $p \in \cl (A) \subs \Sigma_\kappa(S)$ with $S$ arbitrary.
In fact, it is easy to check that $p \in \cl(M\cap A)$ whenever $M$ is a $\kappa$-stable elementary submodel of a big enough $H(\theta)$ such that $p, A \in M$.
It turns out that local tightness distinguishes the $\kappa$-Valdivia algebras from the $\kappa$-Corson ones.

\begin{tw}
\label{Thmweteig}
\label{thm-3.5}
Let $\kappa = \cf \kappa > \aleph_0$ and assume $\Be$ is a $\kappa$-Valdivia Boolean algebra.
Then $\Be$ is $\kappa$-Corson 
if and only if $\ult{(\Be)}$ has local tightness $< \kappa$.

On the other hand, if $\Be$ is $\kappa$-Valdivia but not $\kappa$-Corson, then the interval algebra $\Be(\kappa)$ is a homomorphic image of $\Be$.
\end{tw}

\begin{pf} 
The first part follows from the fact that $\natB$ is $\kappa$-Corson iff $\ult{(\Be)}$ is $\kappa$-Corson and from Theorem~\ref{Thmweteig7878}.

Now suppose $\Be$ is $\kappa$-Valdivia and is not $\kappa$-Corson. 
We may assume that $K \eqdef \ult(\Be)$ is a subset of $\{0,1\}^S$ and that $D \eqdef K \cap \Sigma_\kappa(S)$ is topologically dense in $K$. 
Since $K$ is not $\kappa$-Corson, we choose 
$x \in K \setminus \Sigma_\kappa(S)$.
Let $\lam = |\supp(x)| \goe \kappa$. 

Fix a continuous chain $\sett{M_\al}{\al < \lam}$ of $\kappa$-stable elementary submodels of a big enough $H(\theta)$ such that $x, K \in M_0$, $|M_\al| < \lam$ for $\al < \lam$ and 
$\supp(x) \subs \bigcup_{\al < \lam}M_\al \eqdef M$. 

Fix $\al < \lam$ and set $S_\al = M_\alpha \cap S$. 
We define the function $\fnn{ r_\al(x) }{ S }{ \{0,1\} }$
by the conditions
$$r_\al(x) \rest S_{\al} = x \rest S_\al \oraz r_\al(x) \rest (S \setminus S_\al) = 0.$$
We claim that $r_\al(x) \in K$. 

Let $V$ be a clopen basic neighborhood of $r_\al(x)$ in $2^{S}$. 
So $V$ is determined by a finite set $\sigma \subseteq S$ 
and a map $\fnn{\varphi}{\sigma}{2}$ satisfying 
$r_\al(x) \rest \sigma = \phi$, namely, $V = \setof{z \in 2^S}{z \rest \sig = \phi}$. 
We show that $V \cap K \neq \emptyset$.

Let $\tau \eqdef \sigma \cap S_\alpha = \sigma \cap M_\alpha$ and $\psi = \varphi \restriction \tau$.
Note that $\phi \rest (\sig \setminus \tau) = 0$.
Since $\tau \in M_\al$,  $\psi \in M_\al$ and $\psi$ defines a clopen set $W$ of $2^{S}$. 
We have that $W \in M_\al$.
Since $x \in W \cap K$, by elementarity,   
``$M_\al \models W \cap K \neq \emptyset$''. 
Therefore choose $y \in M_\al \cap W \cap D$. 
So $\supp(y) \subseteq S_\alpha$. 
On the other hand since $y \in M_\al$ and ``$M_\al \models |\supp(y)| < \kappa$'' (\,because $y \in D \subs \Sigma_\kappa(\lam)$\,), by $\kappa$-stability, we obtain that $\supp(y) \subseteq M_\al$.
Consequently, $y \rest (S \setminus S_\al) = 0$ and hence $y \in V$.
Therefore $y \in V \cap K$, showing that $V \cap K \neq \emptyset$. 

Next since $V$ is an arbitrary basic neighborhood of $r_\al(x)$ and $K$ is compact (and thus closed in $2^{S}$), it follows that $r_\al(x) \in K$.

Now, since $\sett{M_\al}{\al<\lam}$ is continuous, and since
$\alpha \leq \beta$ iff $r_\alpha(x) \subseteq r_\beta(x)$ we conclude that $\sett{r_\al(x)}{\al < \lam}$ is homeomorphic to $\lam + 1$. 
In particular $K$ has a topological copy of $\kappa+1$.
\end{pf}

\section{The pointwise topology on a Boolean algebra}
\label{pointwise topology}

Let $\Be$ be a Boolean algebra.
Its Stone space $\ult (\Be)$ can be identified with the space of all homomorphisms of the form $\map h \Be 2$, where $2$ denotes the $2$-element Boolean algebra.
By this way, it is natural to consider the \dfn{pointwise topology}{pointwise topology} on $\Be$, which is the weak topology induced by the family of all homomorphisms into the $2$-element algebra. 

In other words if $K$ is a $0$-dimensional compact space, we consider the Boolean algebra $C(K,2) = \pair{\clop(K)}{\tau_p}$ instead of the space $C(K)$ (endowed with the pointwise topology $\tau_p$).

The standard neighborhood basis consists of finite intersections of the following sets:
$$V^+_q = \setof{a \in \Be}{a \in q} \oraz V^-_q = \setof{a \in \Be}{a \notin q}$$
for $q \in \ult(\Be)$. 
We shall denote this topology by $\tau_p(\Be)$, or shortly, by $\tau_p$.

\medskip

Some applications of Boolean algebras endowed with the pointwise topology appear in the results \ref{ThmJednicka},  \ref{PropDwieapul} and   \ref{ThmDwa}.

\begin{tw}\label{ThmJednicka}
Let $\kappa$ be an uncountable regular cardinal and let $\Be$ be a $\kappa$-Corson Boolean algebra.
Then $\pair \Be {\tau_p}$ is $\kappa$-Lindel\"of. 

In particular, if $\kappa = \lam^+$, then the Lindel\"of number of $\pair \Be {\tau_p}$ does not exceed $\lam$.
\end{tw}

\begin{pf}
Recall that $\Be$ is $\kappa$-Corson iff $K \eqdef  \ult(\natB)$ is $\kappa$-Corson (Corollary~\ref{cor-generation}) and that 
$\clop(\natB) = C(K, 2)  \subseteq C(K)$.
By Theorem~\ref{thm-456789}, $C(K)$ is $\kappa$-Lindel\"of and thus $C(K, 2)$ is $\kappa$-Lindel\"of. 
\end{pf}

\begin{uwgi}
It is clear that the proof above can be adapted to the case of $\kappa$-Valdivia algebras, replacing $\tau_p$ by the topology $\tau_p(D)$ of pointwise convergence on the dense set $D$.
\end{uwgi}

We start with a simple example relevant to our study.

\begin{prop}
Let $\kappa$ be an uncountable cardinal and let $\Be$ be the interval Boolean algebra generated by the chain $\kappa$.
Then $\Be$ is $\kappa$-Corson if and only if $\kappa$ is singular.
\end{prop}

\begin{pf}
Suppose first that $\kappa$ is regular and fix a generating set $G \subs \Be$.
Recall that the space $\ult (\Be)$ is $\kappa + 1$ with the order topology.
We can look at $G$ as a family of clopen subsets of $\kappa + 1$ which separates points, that is, for each $x, y \in \kappa + 1$ there is $a \in G$ satisfying $|a \cap \dn xy| = 1$.

Let $S \subs \kappa$ consist of all $\xi < \kappa$ such that there exists $a \in G$ with $\xi \in a$ and $\kappa \notin a$.
If the set   
$H \eqdef \setof{a \in G}{\kappa \in a}$ has cardinality $\kappa$, then $H$ witnesses the fact that $\natB$ is not $\kappa$-Corson.
Otherwise $|\kappa \setminus S| < \kappa$ and hence $S$ is stationary. 
Therefore $S' \eqdef \setof{ \alpha \in S }{ \alpha \in \kappa \text{ is limit} }$ is stationary. 
For each $\xi \in S'$ choose $a_\xi \in G$ such that $\xi \in a_\xi$
and set $f(\xi) = \min (a_\xi)$. 
Then the function $\map f {S'} \kappa$ is regressive, therefore by Fodor's Pressing Down Lemma there are $\al < \kappa$ and a stationary set $T \subs S'$ such that $f(\xi) = \al$ for $\xi \in T$. 
Thus the set $\setof{a \in G}{\alpha \in a}$ witnesses 
the fact that $\natB$ is not $\kappa$-Corson.

Suppose now $\kappa$ is singular, that is,
 $\kappa = \sup_{\al < \lam}\kappa_\al$, where each $\kappa_\al$ is a cardinal $< \kappa$
and $\lam < \kappa$ is regular.
We may assume that the sequence $\sett{\kappa_\al}{\al < \lam}$ is strictly increasing, continuous (i.e. $\kappa_\delta = \sup_{\xi < \delta}\kappa_\xi$ for every limit ordinal $\delta < \lam$) and $\kappa_0 = 0$.
For each $\al < \lam$ choose a family $G_\al$ of clopen subsets of the interval $[\kappa_\al + 1, \kappa_{\al+1}]$ separating the points of $[\kappa_\al + 1, \kappa_{\al+1}]$ and such that $|G_\al| = \kappa_{\al+1}$.
Define $H = \setof{[\kappa_\al + 1, \kappa]}{\al < \lam}$ and let
$$G = H \cup \hbox{$\bigcup$} \setof{ G_\al }{\al < \lam} \cup \{0\} .$$
Note that
$$|\setof{a\in G}{\kappa \in a}| = |H| = \lam < \kappa.$$
Fix $\xi < \kappa$ and choose $\beta < \lam$ such that $\xi < \kappa_\beta$.
Then $\xi \in a \in G$ implies that either $a \in G_\al$ for some $\al < \beta$ or else $a \in H$.
It follows that
$$|\setof{a \in G}{\xi \in a}| \loe |\beta| \cdot \kappa_\beta < \kappa.$$ 

It remains to check that $G$ separates the points of $\kappa + 1$.
Fix $\xi < \eta \loe \kappa$. 
Let $\beta < \lam$ be minimal such that $\xi \loe \kappa_\beta$.
If $\kappa_\beta < \eta$ then $h \eqdef [\kappa_\beta + 1, \kappa] \in H$ satisfies $\eta \in h$ and $\xi \notin h$.
Next assume $\eta \loe \kappa_\beta$. 

Notice that necessarily $\beta = \al + 1$, because 
$\xi < \eta \loe \kappa_\beta$ and $\sett{\kappa_i}{i < \kappa}$ is continuous and strictly increasing.
We have $\kappa_\al < \xi$ and hence $\xi, \eta \in [\kappa_\al + 1, \kappa_{\al+1}]$.
Thus, there is $a \in G_\al$ such that $|\dn \xi \eta \cap a| = 1$.
This completes the proof. 
\end{pf}

\subsection{Initial chain algebras}
\label{initial}

A relatively big class of Corson-like spaces comes from trees.
Namely, let $\pair T \loe$ be a tree, that is $a_t \eqdef \setof{t' \in T}{t' < t}$ is a well ordering for any $t \in T$. 
In Monk \cite[Ch. 2, Special classes]{M1}, the Boolean subalgebra $\natB$ of $\powerset(T)$ generated by $\setof{ a_t }{ t \in T }$ is called the \dfn{initial chain algebra}{initial chain algebra} over $T$. 
The compact space $\sig T$ was defined and studied by the third author in~\cite{To2},~\cite{To1}.

Fix a tree $T$. 
Define $\sig T$ to be the set of all paths in $T$.
A {\em path} $A$ in a tree $T$ is a linearly ordered set $A \subs T$ such that $\setof{x \in T}{x \leq t} \subs A$ 
whenever $t \in A$. 
Hence any path is an initial subset of a branch and any branch is a path. 
For the pointwise topology, $\sig T$ is closed in $\{0,1\}^T$ 
because ``not being a path" is witnessed by at most two points of $T$.  
So subbasic clopen sets have the form
$$V_t^+ = \setof{A \in \sig T}{t \in A} \quad 
\text{ and } \quad V_t^- = \setof{A \in \sig T}{t \not\in A}$$
where $t \in T$. 
Hence $\setof{V_t^+}{t \in T}$
generates the Boolean algebra of clopen subsets of $\sig T$. 

To $A \in \sigma T$ we associate its characteristic map $\chi_T(A) \subseteq \{ 0,1 \}^T$.  So $| \supp(A) | = | A |$. 
Therefore if any branch of $T$ has cardinality less than $\kappa$ then $\chi_T$ is a topological embedding from $\sigma A$ into $\Sigma_\kappa(T) \subseteq  \{ 0,1 \}^T$. 
Note that $\chi_T$ is an order-isomorphism from $\pair{T}{\leq}$ onto its image. 
On the other hand, assume that $\kappa$ is regular and that $b$ is a a branch of cardinality $\kappa$, then $b$ has an  accumulation point in $\sigma T$ with local tightness $\kappa$. 
Thus, we have proved the following fact (see \cite[Lemma 3]{To1}).

\begin{lemma}
\label{lemma-uiffuidiuiudsf} 
\begin{it}
Let $\kappa$ be a regular cardinal and $T$ be a tree. 
Then the $0$-dimensional compact space $\sigma T$ is $\kappa$-Corson if and only if every branch of $T$ is of cardinality less than $\kappa$.
\qed
\end{it}
\end{lemma}

\begin{remark*}
Define a ``pseudo-tree'' $T$ by requiring that $a_t \eqdef \setof{t' \in T}{t' < t}$ is a linear ordering for every $t \in T$. 
For pseudo-trees, Lemma~\ref{lemma-uiffuidiuiudsf} is not true. 
A counterexample is the linear ordering $\natR$: 
the space $\sigma \natR$ is homeomorphic to $\ult(\natB(\natR))$, and by Proposition~\ref{lemma-jkljlkjlk}, $\natB(\natR)$ is not $\cont$-Valdivia and thus not $\cont$-Corson. 
\end{remark*}

A \dfn{free sequence}{free sequence} in a space $X$ is a sequence $ x \eqdef \seqnn{ x_\xi }{ \xi<\alpha }$ ($\alpha$ an ordinal) of elements of $X$ such that for all $\xi<\alpha$: 
$\cl_X\setof{ x_\eta }{ \eta < \xi } 
\cap \cl_X\setof{ x_\eta }{ \xi \leq \eta < \alpha } = \emptyset$ (recall that $\cl_X$ denote the topological closure operation). 
Any free sequence is a discrete subset of $X$.
Perhaps one of the main interest in free sequences is given by the following theorem due to Arhangelski\u{\i} (based on proofs of Shapirovski\u{\i}): see Monk \cite[Theorem 4.20]{M1} and Juh\'{a}sz \cite[3.12]{J}. 

\begin{lemma}[Arhangelski\u{\i} and Shapirovski\u{\i}]
\label{lemma-hdfjkhs12}
Let $K$ be a compact Hausdorff space. 
Then 
$$\sup  \, \setof{ \tight(p,K) }{p \in K } = \sup \, \setof{\, | \mu | }{\exists \text{ a free sequence in } K \text{ of length } \mu \text{ in } K \,}.\qed$$
\end{lemma}

Following \cite{To1}, for a compact space $K$ we let $\sigma K$ be the set of all free sequences in $K$. 
We order two free sequences 
$x$ and $y$ by $x \preceq y$ if $y$ extends $x$.
Since $\pair{ \sigma K }{ \preceq }$ is a tree, we can consider the topological space $\sigma (\sigma K)$, denoted  by 
$\sigma^2 K$, of all paths of $\sigma K$. 
Our next result is the following.

\begin{prop}
\label{prop-sigma-1}
Let $\kappa$ be an uncountable regular cardinal and $K$ be a compact Hausdorff space $K$. 
The space $\sigma^2 K$ is $\kappa$-Corson if and only if $K$ has local tightness $< \kappa$. 
\end{prop}

\begin{proof}
Note that for any free sequence $x$ in $K$, i.e. $x \in \sigma K$,  the path 
$\setof{y \in \sigma K}{ y \preceq x } \in \sigma^2 K$ is a well-ordering of cardinality $ | x |$. 

Now by Proposition~\ref{lemma-uiffuidiuiudsf} , $\sigma^2 K$ is $\kappa$-Corson iff every branch of $\sigma K$ is of cardinality less than $\kappa$ iff every path of $\sigma^2 K$ is of cardinality less than $\kappa$. 
On the other hand by Lemma~\ref{lemma-hdfjkhs12}, 
every free sequence $x$ in $K$ is of length less than $\kappa$, i.e. every member of $\sigma K$ is of cardinality less than $\kappa$, iff $\tight(p, K) < \kappa$ for every $p \in K$. 
\end{proof}

We now develop the above result in terms of Boolean algebras, see:~\cite{M2}.
Fix a Boolean algebra $\Be$.
A \dfn{free sequence}{free sequence} in $\Be$ is a sequence $\sett{a_\xi}{\xi < \alpha}$, where $\alpha$ is an ordinal, and for every finite sets $S, T \subs \alpha$ such that $\xi < \zeta$ whenever $\xi \in S$ and $\zeta \in T$, it holds that 
\vspace{-1.5mm}
$$\hbox{$\prod$}_{\xi \in S} \,\, a_\xi \,\, \cdot \,\, \hbox{$\prod$}_{\zeta \in T} - a_\zeta > 0$$ 
(Finite intersection property).
This notion of a free sequence is closely related to the usual notion of a free sequence of points in a topological space. 
Let $\Be$ be a Boolean algebra. 
\begin{itemize}
\item[{\rm($*$)}]
For any ordinal $\alpha$, there is a free sequence $\seqnn{ a_\xi }{ \xi < \alpha }$ in $\Be$ if and only if there is a free sequence $\seqnn{ x_\xi }{ \xi<\alpha }$ in $\ult(\Be)$.
\end{itemize}
%
%
Indeed let $\seqnn{ a_\xi }{ \xi<\alpha }$ be a free sequence in the algebra $\Be$. 
For each $\beta<\alpha$, by compactness, choose $x_\beta \in \ult(\Be)$ such that $\prod_{\xi < \beta} a_\xi \, \cdot \, \prod_{\beta \leq \zeta < \alpha} - a_\zeta \in x_\beta$. 
Then $ \seqnn{ x_\xi }{ \xi<\alpha }$ is a free sequence over $\ult(\Be)$. 
Conversely let  $\seqnn{ x_\xi }{ \xi<\alpha }$ be a free sequence in $\ult(\Be)$. 
For each $\beta < \alpha$ choose $a_\beta \in \Be$ such that 
%
$\setof{ x_\xi }{ \xi \leq \beta } \subseteq 
\setof{ y \in \ult(\Be)  \,  }{ \, {-}a_\beta \in y }$ and 
%
$\setof{ x_\xi }{ \xi > \beta } \subseteq 
\setof{ y \in \ult(\Be)  \,  }{  \, a_\beta \in y }$.  
%
%
Then $ \seqnn{ a_\xi }{ \xi<\alpha }$ is a free sequence over $\Be$.
These processes are not inverses of each other.

\smallskip

Let $\natB$ be a Boolean algebra. 
Denote by $\sigma(\Be)$ the tree of all free sequences in $\Be$, ordered in the usual sense, that is, $p \loe q$ if $q$ extends $p$. 
Since $\pair{ \sigma(\Be) }{ \preceq }$ is a tree we can consider the set $\sigma (\sigma(\Be)) \subseteq \{0,1\}^{\sigma(\Be)}$, denoted  by 
$\sigma^2 \Be$, of all paths of $\sigma(\Be)$ endowed with the pointwise topology. 
Recall that for an uncountable cardinal $\kappa$ and for a Boolean algebra $\natA$: $\ult(\natA)$ is $\kappa$-Corson if and only if $\natA$ is $\kappa$-Corson (Corollary~\ref{cor-generation}). 
The proof of the following result is similar to that of Proposition~\ref{prop-sigma-1}.

\begin{prop}
\label{prop-sigma-2}
Let $\kappa$ be an uncountable regular cardinal.
Given a Boolean algebra $\Be$, the space $\sig^2\Be$ is $\kappa$-Corson if and only if $\ult(\Be)$ has local tightness $< \kappa$. 
\qed
\end{prop}

Setting 
$$\tight(\Be) = \sup \setof{ |\alpha| }{ \text{\rm there is a free sequence in } \Be \text{\rm\ of order type } \alpha },$$
by ($*$), we have that
$\tight(\Be) = \tight(\ult(\Be))$.

\medskip

We do not know an example where the spaces $\sigma^2(\Be)$ and $\sigma^2(\ult(\Be))$ are not homeomorphic, but it is not difficult to see that $\sig^2(\Be)$ is homeomorphic to a subspace of $\sigma^2(\ult(\Be))$.

\medskip

Recall that a Boolean algebra $\Be$ has the \dfn{$\lambda$-separation property}{separation property}, if for each $A, B \subs \Be$ such that $|A| < \lambda$, $|B| < \lambda$, satisfying $a \cdot b = 0$ for all $a \in A$, $b \in B$, there exists $c \in \Be$ such that $a \loe c$ for $a \in A$ and $b \loe - c$ for $b \in B$. 
So the countable separation property is the $\aleph_1$-separation property. 

Remark that if $\Be$ is $\lambda$-complete (every subset of cardinality less than $\lambda$ has a supremum and an infimum) then $\Be$ has the $\lambda$--separation property.

Take $K$ to be the Stone space of a complete Boolean algebra 
$\Be$ of size $\cont$ (for instance $\Be = \powerset(\natN$) and choose $D \subs K$ so that $\abs{D} = \cont^+$.
Then $\alex(K,D)$ is $\cont^+$-Corson and it still has the countable separation property. 
Indeed since $| \Be | < \cont^+$ the space $K$ is $\cont^+$-Corson and thus, by Proposition~\ref{Prowngog}, $\alex(K,D)$ is $\cont^+$-Corson. 
Next, since $\Be$ is $\sigma$-complete, it has the countable separation property, therefore so does $\alex(K,D)$.

\smallskip

As usual if $\lambda$ is a cardinal, we denote by $\lambda^+$ its successor cardinal and any cardinal is considered as an ordinal. 

\begin{corollary} 
\label{cor-cor-456456456}
Assume $\lambda^{<\lambda} = \lambda$.
There exists a $\lambda^+$-Corson Boolean algebra $\Be$ with the $\lambda$-separation property and such that $|\Be| = \lambda^+$.
\end{corollary}
\begin{proof} 
We consider the set $K = \{0,1\}^{\lambda}$, of all binary $\lambda$-sequences, ordered lexicographically and endowed with the interval topology. 
We denote by $L$ the subset of $K$ consisting of all $x \in K$ such that $\supp(x) \eqdef \setof{ \xi < \lambda }{ x(\xi) \neq 0 }$ is of cardinality $ < \lambda$. 
Note that $| L | = \lambda^{<\lambda} = \lambda < 2^{\lambda} = | K |$. 
This linear ordering $L$ was introduced by Hausdorff and is usually denoted by $\eta_\alpha$, see~\cite[Theorem 9.25, Ch. 9, \S4]{Ros}.

Let $\Be_0$ be the interval algebra over the linear ordering $L$.  
So $|\Be_0| = |L| = \lambda$ and thus $\Be_0$ is $\lambda^+$-Corson.
Now choose $D \subs K$ so that $|D| = \lambda^+ < |K|$ 
and let $\Be$ be the algebra  of clopen sets of the Alexandrov duplication $\alex(\ult(\Be_0), D)$. 
By Proposition~\ref{Prowngog}, $\Be$ is $\lambda^+$-Corson, 
and obviously $| \Be | = \lambda^+$.
Next, since $L$ contains neither a copy of the ordinal $\lambda^+$ nor its converse, therefore $\Be_0$ has the $\lambda$-separation property. 
\end{proof}

In Theorem~\ref{thm-uihkjvv} we shall see that a complete Boolean algebra of regular size $\kappa > \aleph_0$ cannot be $\kappa$-Corson.

\subsection{Poset algebras and modest semilattice algebras }
\label{SectModestlats}

Let $P$ be a poset. 
We say that $F \subseteq P$ is a \dfn{final segment}{final segment} whenever 
$\upclcone{p} \eqdef \setof{ q \in P }{ q \geq p } \subseteq F$ for every $p \in F$. 
Also in a lattice $L$ a set $F \subseteq L$ is a 
\dfn{filter}{filter} if $F$ is a final subset of $L$ such that 
$x \wedge y \in F$ for every $x,y \in F$, and  
a filter is \dfn{prime}{prime filter} if moreover $x \join y \in F$ implies $x \in F$ or $y \in F$.

We define the \dfn{poset algebra\/}{poset algebra} $\Be(P)$ 
as follows.
Let $\Fr(P)$ be the free Boolean algebra over the set $P$ 
and $\map{x}{ P }{ \Fr(P) }$ be the canonical embedding from $P$
into $\Fr(P)$. 
Let $I_P$ be the ideal of $\Fr(P)$ generated  the set
$\setof{ (x(p) \cdot x(q))  \symdiff  x(p) }{ p \leq q }$, 
that is by the set $\setof{ x(p) - x(q) }{ p \leq q }$. 
We set $\Be(P) = \Fr(P) / I_P$.  
We denote $x(p) / I_P$ by $x_p$. 
The map $p \mapsto x_p$ is one-to-one and 
$p \leq q$ in $P$ if and only if $x_p \leq x_q$ in $\Be(P)$.

We denote by \dfn{$\rfs{FS}(P)$}{$\rfs{FS}(P)$} the set of all final segments of $P$. 
So $\emptyset , P \in \rfs{FS}(P)$.
That is (up to characteristic function) the set of all increasing maps $\chi$
from $P$ into $\{0,1\}$. 
So $\rfs{FS}(P)$ is a closed subspace of $2^{\powerset(P)}$ 
(for the pointwise topology).
Notice that $\fourthple{ \rfs{FS}(P) }{ \cup }{ \cap }{ \tau }$
is a distributive lattice and that, 
with the pointwise topology $\tau_p$, the meet and the union operations are continuous. 
Trivially, for every $p \in P$ the set 
$a_p \eqdef \setof{ u \in \rfs{FS}(P) }{ p \in u }$ 
is a clopen final subset of $\rfs{FS}(P)$. 
It is easy to check that the map $x_p \mapsto a_p$ is extendable to a Boolean isomorphism 
from $\Be(P)$ onto $\Clop(\rfs{FS}(P))$. 
Hence we identify $\Ult(\Be(P))$ and $\rfs{FS}(P)$. 
(See Abraham, Bonnet, Kubi\'s and Rubin~\cite[\S2]{ABKR}.)

Note that every interval algebra is a poset algebra. 
\begin{prop}\label{PropDwieapul}
Every poset Boolean algebra contains a generating set that is closed and discrete with respect to the pointwise topology.
\end{prop}

\begin{pf}

Let $\Be$ be a poset Boolean algebra over $P$. 
Note that $G \eqdef \setof{ a_p }{ p \in P }$ generates $\Be$. 
For $p \in P$, the set $a_p$ is a clopen prime filter of $\rfs{FS}(P)$. 
Conversely if $F \subseteq \rfs{FS}(P)$ is a clopen and prime filter,
then $F$ has a minimum of the form $ [p, \rightarrow)$
with $p \in P$, and thus  $F = a_{p}$ \cite[Lemma 2.9]{ABKR}. 

We show that $G$ is discrete in $\pair \Be {\tau_p}$. 
Fix $p \in P$.
So $J_p \eqdef \rfs{FS}(P) \setminus a_p$ is a clopen ideal of $\rfs{FS}(P)$ and thus $J_p$ has a maximum defined by a finite subset $\tau_p$ of $P$. 
Therefore $a_p$ is the unique clopen filter of $\rfs{FS}(P)$ belonging to the clopen set 
$\setof{ u \in \rfs{FS}(P) }
{ p \in u \text{\rm\ and } u \cap \tau_p = \emptyset }$ of $\rfs{FS}(P)$. 

Next $G$ is closed in $\pair \Be {\tau_p}$, because ``not being a prime filter in $\rfs{FS}(P)$" is witnessed by at most three points of $P$ 
(as a clopen set in $\rfs{FS}(P)$ is defined by finitely many points of $P$). 
Furthermore, recall that $G$ generates the Boolean algebra $\Be$. In particular, $|G| = |\Be|$.
\end{pf}

Next, let $\natM = \pair{M}{\wedge}$ be a semilattice.
We define the \dfn{semilattice algebra\/}{semilattice algebra} $\Be(M)$ 
as follows.
Let $\Fr(M)$ be the free Boolean algebra over the set $M$.
Recall that $x(p)$ denotes $x_p$.
Let $J_M$ be the ideal of $\Fr(M)$ generated  the set
$\setof{ x_p \cdot x_q  \symdiff  x_{p \wedge q} }{ p , q \in M }$. 
We set $\Be(M) \eqdef \Fr(M) / J_M$. 
We denote by $\rfs{Fil}(M)$ the set of all filters of $M$.
In particular $\emptyset \in \rfs{Fil}(M)$. 
Note that $\rfs{Fil}(M)$ corresponds (via characteristic functions) to the set of semilattice preserving  maps from $M$ into $\{0,1\}$.
Hence $\rfs{Fil}(M)$ is a closed subspace of $2^{\powerset(M)}$ 
(for the pointwise topology). 

For $p \in M$ we set $a_p \eqdef \setof{ u \in \rfs{Fil}(M) }{ p \in u }$.
It is easy to check that the map $x_p \mapsto a_p$ is extendable to a Boolean isomorphism 
from $\Be(M)$ onto $\Clop(\rfs{Fil}(M))$. 
Hence we identify $\Ult(\Be(M))$ and $\rfs{Fil}(M)$. 
Also $\fourthple {\rfs{Fil}(M)} \cap \emptyset  \tau$ 
is a compact $0$-dimensional semilattice and $\cap$ is continuous.

Let $\fourthple L \meet 0 \tau$ be a 
\dfn{compact $0$-dimensional semilattice}{}. 
Recall that, by the definition, the meet operation $\meet$ is continuous for the topology $\tau$. 
Since $\meet$ is continuous, $L$ has a smallest element $0$.
An element $a \in L$ is \dfn{compact}{compact} if, by definition, $a > 0$ and the following implication holds for every set $S \subs L$:
$$\sup S = a \implies (\exists\; F \in \fin S)\;\; \sup F = a.$$
This clearly justifies the name ``compact".
The property of being a compact element turns out to be equivalent to the fact that the filter $[a ,\rightarrow)$ is a clopen subset of $L$.
We denote by $K(L)$ the set of compact elements of $L$. 
So $0 \not\in K(L)$.
Then $\clop(L)$ is the Boolean algebra generated by 
$\setof{ \upclcone{p} }{ p \in K(L)}$. 
Also if $p$, $q$ are elements of $K(L)$ then 
either $\upclcone{p} \cap \upclcone{q}$ is empty 
or $\upclcone{p} \cap \upclcone{q} = \upclcone{s}$, 
where $s \eqdef p \vee^L q \in K(L)$. 
So $L = \ult(\Be(M))$ where $M = K(L) \cup \{0\}$
and for every $p,q \in K(L)$, either $p \meet^M q = s$ if 
$s \eqdef p \vee^L q$ exists,
nor $p \meet^M q = 0$ (that corresponds to $\upclcone{p} \cap \upclcone{q} = \emptyset$).

We refer to the book of Hofmann,  Mislove, and  Stralka \cite{HMS} for much more detailed information concerning compact semilattices.

A compact 0-dimensional semilattice $\triple L \meet 0$ is \dfn{modest}{modest (semi lattice)} if for every compact element $a \in L$ the set of its immediate predecessors is finite. 
So, in a modest semilattice $L$, for any compact element $a$, 
there exists a finite set $I_a$ such that $c < a$ for every $c \in I_a$, and for every $x < a$ there exists $c \in I_a$ such that $x \loe c$.
The notion of a modest semilattice was introduced recently in~\cite[\S3]{KMT}.

Basic non-trivial examples of modest semilattices are compact 0-dimensional linearly ordered spaces (every element has at most one immediate predecessor).
More general examples are median spaces (see for instance \cite{BH}, \cite{vV}), treated as semilattices.
In particular, every 0-dimensional compact distributive lattice is a modest semilattice (that corresponds to the space of ultrafilters of a poset algebra
\cite[Theorem 2.6]{ABKR}).

Given a subset $A$ of a partially ordered set, we denote by $\max A$ the set of all maximal elements of $A$. 
So $\max A$ can be empty.

\begin{tw}\label{ThmDwa}
Let $\Be$ be an infinite Boolean algebra such that the space $L = \ult(\Be)$ carries a structure of a modest topological semilattice.
Assume further that there exists a set of cardinality $< |\Be|$ which is topologically dense in $\max L$. 

Then for every regular cardinal $\lam \loe |\Be|$ there exists a closed discrete set in $\pair \Be {\tau_p}$ of cardinality $\lam$.

In particular, if $|\Be|$ is an uncountable regular cardinal, then for every generating set $G \subs \Be$ there exists $p\in \ult(\Be)$ such that $|p \cap G| = |\Be|$, that is $\Be$ is not $\kappa$-Corson.
\end{tw}

\begin{pf}
We identify $\Be$ with the algebra of clopen subsets of $L$, so that we shall write $p \in a$ having in mind that $a\in \Be$ is an element of $p \in \ult(\Be)$.

We use some ideas from duality theory for semilattices.
Namely, denote by $G$ the set of all clopen filters of $L$.
Thus, $\emptyset \in G$ and nonempty elements of $G$ are of the form $[a, \rightarrow)$, where $a > 0$ is a compact element.
Straight from the definition of a compact element, knowing that $L$ is modest, it follows that for each compact element $a$ the filter $[a,\rightarrow)$ is isolated in $G$.
The only possibly non-isolated element of $G$ is the empty set $\emptyset$. 

Note that $G$ is closed in $\pair \Be {\tau_p}$, because ``not being a filter" is witnessed by at most three points of $L$ and $K(L) \cup \{0\}$ is closed under meet.
Furthermore, recall that $G$ generates the Boolean algebra $\Be$.
In particular, $|G| = |\Be|$. 

Fix $D \subs \max L$ such that $|D| < |\Be|$ and $D$ is dense in $\max L$.
Since $L$ is a compact space, for every $x\in L$ there is $y \in \max L$ such that $x \loe y$.
In particular, if $a \in L$ is a compact element, then $[a,\rightarrow) \cap \max L$ is nonempty and open in $\max L$, therefore there exists $d\in D$ such that $a \loe d$.
Fix a regular cardinal $\lam \leq |\Be|$; possibly replacing 
it by $|D|^+$, we assume that $|D| < \lam$.
By the previous remark, there exists $p \in D$ such that the interval $[0,p]$ contains at least $\lam$ compact elements.
Let $V_p = \setof{u \in G}{p \notin u}$.
Then $V_p$ is a neighborhood of $\emptyset$.
Thus $G \setminus V_p$ is a closed discrete set of cardinality at least $\lam$.

The second statement follows from Theorem~\ref{ThmJednicka}.
\end{pf}

Note that, under the assumptions of the theorem above, 
if $\max L$ is nonempty and finite, and if for every $x \in L$
there is $m \in \max L$ such that $x \leq m$,
then the set $G$ of all clopen filters is discrete and closed
with respect to the pointwise topology on $\Be$.  
As we have mentioned already, poset Boolean algebras satisfy the assumptions of Theorem~\ref{ThmDwa}.
On the other hand, for these algebras, we have improved the result in Proposition \ref{PropDwieapul}.

\subsection{Interval algebras}
\label{section-interval}

From the results of the previous section we obtain the following property of poset Boolean algebras, which might be of independent interest.

\begin{tw} 
\label{ThmTrojkaa}
Let $\Be$ be a poset Boolean algebra and let $\kappa$ be a regular uncountable cardinal such that $|\Be| \goe \kappa$.
Then for every generating set $G \subs \Be$ there exists an ultrafilter $p \in \ult (\Be)$ such that $|G \cap p| \goe \kappa$  and thus $\Be$ is not $\kappa$-Corson.

In particular, if $\abs{\Be}$ is a regular uncountable cardinal then for every generating set $G \subs \Be$ there exists an ultrafilter $p \in \ult (\Be)$ such that $\abs{ G \cap p } = \abs{\Be}$, that is $\Be$ is not $\kappa$-Corson.
\end{tw}

\begin{pf}
Suppose the statement above is not true.
Then $\Be$ is $\kappa$-Corson, therefore by Theorem~\ref{ThmJednicka} every open cover of $\pair \Be {\tau_p}$ has a subcover of cardinality $< \kappa$.
On the other hand, by Proposition~\ref{PropDwieapul}, the algebra $\Be$ contains a closed discrete set of cardinality $\kappa$.
This is a contradiction.
\end{pf}

A variant of the result~\ref{ThmTrojkaa}, involving a topologically dense set of ultrafilters is not true.
A counterexample is any uncountable free Boolean algebra.
It turns out however that for some interval algebras one can still prove something, generalizing Proposition~\ref{lemma-jkljlkjlk}:

\begin{tw}
\label{thm-valdivia-interval}
Assume $C$ is a chain of regular cardinality $\kappa > \aleph_0$ such that
each non-trivial interval of $C$ has cardinality $\kappa$ and
no subset of $C$ is order isomorphic to $\kappa$ or its inverse.
Then for every generating set $G$ of the interval algebra $\Be(C)$
the set
$$S = \setof{p \in \ult(\Be(C))}{\abs{G \cap p} < \kappa}$$
is nowhere dense in $\ult(\Be(C))$.
\end{tw}

\begin{pf}
Suppose $G \subs \Be(C)$ is a counterexample and let $K = \ult(\Be(C))$.
We can find a non-trivial interval $[s,t] \subs K$ such that the set $S\cap[s,t]$ is topologically 
is dense in $[s,t]$.
We conclude that the space $L \eqdef [s,t]$ is $\kappa$-Valdivia compact.
By the assumption on $C$, we have that $\tight(p,A) < \kappa$ for every $A \subs K$ and for every $p \in \cl(A)$.
In other words, $K$ and thus $L$, 
has local tightness $<\kappa$.
By Theorem~\ref{Thmweteig}, it follows that $L$ is $\kappa$-Corson.
On the other hand, the weight of $L$ is $\kappa$ because all non-trivial intervals of $C$ have size $\kappa$ 
This contradicts Theorem~\ref{ThmTrojkaa}.
\end{pf}

Recall that for a set $A \subs \ult (\Be)$, in \S\ref{Corson-like-BA}, (\ref{star}), we have seen that: 
$$\cl (A) = \setof{p\in \ult (\Be)}{p \subs \bigcup A}.$$

\begin{lm}\label{LeFrsrw}
Let $\Be$ be a Boolean algebra with a generating set $G$ and let $\lam$ be an infinite cardinal.
Let
$$D_\lam = \setof{p\in \ult (\Be)}{ |G \cap p| \loe \lam }.$$
Then
\begin{enumerate}
	\item[{\rm(a)}] The closure of every subset of $D_\lam$ of cardinality $\loe \lam$ is contained in $D_\lam$. 
	\item[{\rm(b)}] The tightness of $D_\lam$ does not exceed $\lam$.
\end{enumerate}
\end{lm}

\begin{pf}
(a) 
follows from the following fact.
Assume $p \in \cl (A)$ with $A\subs D_\lam$ and $|A| \loe \lam$.
Then
$$G \cap p \subs G \cap \bigcup A = \bigcup_{x\in A}(G \cap x),
$$
therefore $|G \cap p| \loe \lam \cdot |A| \loe \lam$.

(b) Assume $p \in \cl (A) \setminus A$, where $A \subs D_\lam$.
Fix a big enough regular cardinal $\theta$ and fix an elementary substructure $M$ of $\pair {H(\theta)}\in$ such that $p, A \in M$, $\lam+1 \subs M$ and $|M| = \lam$.
We claim that $p \in \cl(A \cap M)$.

Suppose otherwise and choose finite sets $S, T \subs G$ such that, letting
$$u := \hbox{$\prod$} S \cdot - \hbox{$\sum T$},$$
we have that $u \in p$ and $u \notin \bigcup A$. 
Note that $S \subs G \cap p \subs M$, because $p\in M$, $|G \cap p| \leq \lambda$ and $M$ contains all elements of sets of cardinality $\loe \lam$.
Let $T' = T \cap M$ and $u' = \prod S \cdot - \sum T'$. 
Then $u' \in p$ and also $u' \in M$, because $u \leq u'$  
and $S, T' \subseteq G \cap p \subseteq M$. 
Note that $$M \models p \subs \bigcup A,$$
therefore, by elementarity,  using the fact that $u' \in p \cap M$, there is $q \in A \cap M$ 
such that $u' \in q$. 
In particular $S \subs q$ and $T' \cap q = \emptyset$,   
because $q$ is an ultrafilter of $\Be$ containing $u'$.
On the other hand, since $u \notin \bigcup A$, we have 
$u \notin q$, therefore $T \cap q \nnempty$, because $S \subs q$.
Finally, $T \cap q \subs G \cap q \subs M$, 
and thus $T' \cap q = T \cap q \neq \emptyset$, a contradiction.
\end{pf}

\begin{tw}\label{ThmTrojka}
Let $\Be$ be an interval Boolean algebra and let $\lam$ be an infinite cardinal such that $|\Be| > \lam^+$.
Then for every generating set $G \subs \Be$, for every topologically dense set $D \subs \ult (\Be)$, there exists $p \in D$ such that $|G \cap p| > \lam$.
\end{tw}

\begin{pf}
Suppose $G \subs \Be$ is a generating set and $D \subs \ult (\Be)$ is a topologically dense set such that $|G \cap p| \loe \lam$ for every $p \in D$.
We shall show that the tightness of $\ult (\Be)$ does not exceed $\lam^+$.
After that, using Theorem~\ref{Thmweteig}, we deduce that $\Be$ is $\lam^{++}$-Corson and finally, by Theorem~\ref{ThmTrojkaa}, 
we conclude that $|\Be| \loe \lam^+$. 
This is a contradiction.

Thus, it remains to show that the tightness of $\ult (\Be)$ is $\loe \lam^+$. 
This is equivalent to the fact that whenever $C$ is a generating chain for $\Be$ then the ordinal $\lam^{++}$ does not embed into $C$.
Suppose otherwise and choose a strictly increasing monotonic sequence $\sett{c_\xi}{\xi < \lam^{++}} \subs C$. 
We may assume that 
$c_0 = 0^{\natB}$ and that $\sett{c_\xi}{\xi < \lam^{++}}$ is strictly increasing. 

For each $\xi < \lam^{++}$ choose $p_\xi \in D$ so that $c_{\xi+1} \in p_\xi$ and $c_\xi \notin p_\xi$.
Let $p \in \ult( \Be)$ be an ultrafilter containing the set
$$\setof{- c_\xi}{\xi < \lam^{++}}
\cup
\setof{ c \in C }{ c > c_\xi \text{ for every } \xi < \lam^{++} }.
$$
This actually defines $p$ uniquely. 
Note that for every $a \in p$, there is $\alpha(a) < \lam^{++}$ such that $a \in p_\xi$ for every $\xi > \alpha(a)$. 
On the other hand, given $\al < \lam^{++}$, we have that $p \notin \cl (\setof{p_\xi}{\xi < \al})$, which is witnessed by the fact that $- c_\alpha \in p$.

Now observe that $|G \cap p| > \lam^+$.
Indeed, otherwise by Lemma~\ref{LeFrsrw}(b), there would exist $S \subs \lam^{++}$ of cardinality $\loe \lam^+$ such that $p \in \cl (\setof{p_\xi}{\xi \in S})$, which is impossible by the regularity of $\lam^{++}$.

Choose $H \subs G \cap p$ such that $|H| = \lam^+$.
For each $h \in H$ find $\al(h) < \lam^{++}$ such that $h \in p_\xi$ for $\xi \goe \al(h)$.
Let $\beta = \sup_{h \in H} \al(h)$.
Then $\beta < \lam^{++}$ and $H \subs G \cap p_\beta$. 

On the other hand, $p_\beta \in D$, which means that  $|G\cap p_\beta| \loe \lam$, a contradiction.
This completes the proof. 
\end{pf}

As a consequence of previous results, we obtain the following:

\begin{tw}
\label{thm-uihkjvv}
Let $\Be$ be a complete Boolean algebra of an uncountable regular cardinality $\kappa$.
Then $\Be$ is not $\kappa$-Corson, that is, given a generating set $G \subs \Be$ there exists an ultrafilter $p$ such that $\abs {p \cap G} = \abs{\Be}$. 
\end{tw}

\begin{pf}
Let $\kappa = \abs{\Be}$.
By the Theorem of Balcar and 
Franek~\cite{BF}, the free Boolean algebra $\Fr(\kappa)$ embeds into $\Be$.
Obviously,  $\Fr(\kappa)$ is not $\kappa$-Corson (this also follows from Theorem~\ref{ThmTrojkaa}).
Finally, by Corollary~\ref{Thmjgwrgo}, we conclude that $\Be$ is not $\kappa$-Corson.
\end{pf}

\begin{tw}
\label{thm-2.14}
\label{thm-4.17}
Assume $\Be$ is an infinite Boolean algebra with the countable separation property.
Let $\kappa$ be an uncountable regular cardinal. 
If $\Be$ is $\kappa$-Valdivia then it is $\kappa$-Corson.
\end{tw}

\begin{pf}
Otherwise, by Theorem~\ref{Thmweteig}, $\Be$ has a homomorphism onto the interval algebra $\Be(\kappa)$ and thus onto $\Be(\omega+1)$, which contradicts the fact that $\Be$ has the countable separation property.
\end{pf}


\section{Concluding remarks and open questions}
\label{question}

\begin{pyt}
Let $K$ be a $\kappa$-Corson space. 
Is its hyperspace $H(K)$ $\kappa$-Corson?
Here $H(K)$ is the space of all nonempty compact subsets endowed with the Vietoris topology.
\end{pyt}

Negative answer has been found by the Referee; the arguments come from Bell~\cite{Bell}, although the following claim probably belongs to the folklore.

\begin{prop}
	Let $\kappa$ be an uncountable cardinal and let $A(\kappa)$ denote the one-point compactification of the discrete space of cardinality $\kappa$. Then $H(A(\kappa))$ contains a copy of the Cantor cube $2^\kappa$.
\end{prop}

\begin{pf}
	Define $\map \phi {\Pee(\kappa)}{H(A(\kappa))}$ by setting $\phi(x) = x \cup \sn \infty$, where $\infty$ denotes the unique accumulation point of $A(\kappa)$. It is easily seen that $\phi$ is a homeomorphic embedding, when $\Pee(\kappa)$ is viewed as the Cantor cube $2^\kappa$.
\end{pf}

Note that $A(\kappa)$ is Corson (in particular, $\kappa$-Corson), while $2^\kappa$ is not $\kappa$-Corson, due to Corollary~\ref{cor-ieyzuirzeuirez}.

The following natural question has been recently answered in the negative by Plebanek~\cite{Ple}.

\begin{pyt}
Let $\Be$ be a $\sig$-complete Boolean algebra of cardinality $\kappa > \aleph_0$.
Is it possible that $\Be$ is $\kappa$-Corson?
\end{pyt}

\koment{
\begin{pyt}[Kalenda]
Let $\Be = \Fr(\kappa)$, the free Boolean algebra with $\kappa > \aleph_0$ generators.
Assume $\kappa > \cf \kappa$.
Is it possible that $\Be$ is $\kappa$-Corson?
\end{pyt}
}

In the next question we consider Boolean algebras with the pointwise topology.

\begin{pyt}
(1)
Characterize Boolean algebras for which there is a discrete and closed set of generators.

(2)
Characterize Boolean algebras for which the Lindel\"of number is equal to
the cardinality of the Boolean algebra.
\end{pyt}

\begin{pyt}
	Are the results above (e.g. Theorem~\ref{thm-753}, Corollary~\ref{cor-hfdjskhjkf}, Theorem~\ref{thm-456789}) valid for singular cardinals?
\end{pyt}


\printindex

\end{document}